\documentclass[a4paper,reqno,11pt]{amsart}
\usepackage{stmaryrd}
\usepackage{mathtools}
\usepackage{dsfont}
\usepackage{amsmath, amsfonts, amssymb, amsthm, amscd}
\usepackage{graphicx}
\usepackage{psfrag}
\usepackage{perpage}
\usepackage{url}
\usepackage{mathrsfs}
\usepackage{tikz}
\usepackage{mathabx}
\usetikzlibrary{arrows,decorations.pathmorphing,backgrounds,positioning,fit,petri} 
\usepackage{tikz,tikz-cd}\usepackage{caption,subcaption}

\usepackage{dsfont} 
\usepackage{esint}

\usepackage[utf8]{inputenc}
\usepackage[T1]{fontenc}
\usepackage{microtype}

\usepackage[a4paper,scale={0.72,0.74},marginratio={1:1},footskip=7mm,headsep=10mm]{geometry}

\usepackage[hidelinks]{hyperref}

\usepackage{amsmath}

\newcommand{\avgint}{%
	\mathop{%
		\vcenter{\hbox{%
				$\int$\kern-0.7em\raisebox{0.6ex}{\rule{0.8em}{0.1pt}}%
		}}%
	}%
}

\makeatletter
\def\@secnumfont{\bfseries\scshape}

\def\section{\@startsection{section}{1}%
	\z@{.7\linespacing\@plus\linespacing}{.5\linespacing}%
	{\normalfont\large\bfseries\scshape\centering}}

\def\subsection{\@startsection{subsection}{2}%
	\z@{.5\linespacing\@plus.7\linespacing}{-.5em}%
	{\normalfont\bfseries\scshape}}

\def\subsubsection{\@startsection{subsubsection}{3}%
	\z@{.5\linespacing\@plus.7\linespacing}{-.5em}%
	{\normalfont\scshape}}

\def\specialsection{\@startsection{section}{1}%
	\z@{\linespacing\@plus\linespacing}{.5\linespacing}%
	{\normalfont\centering\large\bfseries\scshape}}
\makeatother

\makeatletter

\renewenvironment{proof}[1][\proofname]{\par
	\pushQED{\qed}%
	\normalfont \topsep4\p@\@plus4\p@\relax
	\trivlist
	\item[\hskip\labelsep
	\bfseries
	#1\@addpunct{.}]\ignorespaces
}{%
	\popQED\endtrivlist\@endpefalse
}
\makeatother

\makeatletter
\newcommand \Dotfill {\leavevmode \leaders \hb@xt@ 6pt{\hss .\hss }\hfill \kern \z@}
\makeatother

\makeatletter
\def\@tocline#1#2#3#4#5#6#7{\relax
	\ifnum #1>\c@tocdepth 
	\else
	\par \addpenalty\@secpenalty\addvspace{#2}%
	\begingroup \hyphenpenalty\@M
	\@ifempty{#4}{%
		\@tempdima\csname r@tocindent\number#1\endcsname\relax
	}{%
		\@tempdima#4\relax
	}%
	\parindent\z@ \leftskip#3\relax \advance\leftskip\@tempdima\relax
	\rightskip\@pnumwidth plus4em \parfillskip-\@pnumwidth
	#5\leavevmode\hskip-\@tempdima
	\ifcase #1
	\or\or \hskip 1.65em \or \hskip 3.3em \else \hskip 4.95em \fi%
	#6\nobreak\relax
	\Dotfill
	\hbox to\@pnumwidth{\@tocpagenum{#7}}\par
	\nobreak
	\endgroup
	\fi}
\makeatother

\makeatletter
\def\l@section{\@tocline{1}{0pt}{1pc}{}{\scshape}}
\renewcommand{\tocsection}[3]{%
	\indentlabel{\@ifnotempty{#2}{\ignorespaces#1 #2.\hskip 0.7em}}#3}
\def\l@subsection{\@tocline{2}{0pt}{1pc}{5pc}{}}

\def\l@subsubsection{\@tocline{3}{0pt}{1pc}{7pc}{}}

\makeatother

\numberwithin{equation}{section}

\newtheoremstyle{mytheorem}{.7\linespacing\@plus.3\linespacing}{.7\linespacing\@plus.3\linespacing}%
{\itshape}
{}
{\bfseries}
{. }
{0.3ex}
{\thmname{{\bfseries #1}}\thmnumber{ {\bfseries #2}}\thmnote{ (#3)}}  

\theoremstyle{mytheorem}

\newtheorem{theorem}{Theorem}[section]
\newtheorem{lemma}[theorem]{Lemma}
\newtheorem{proposition}[theorem]{Proposition}
\newtheorem{corollary}[theorem]{Corollary}
\newtheorem{remark}[theorem]{Remark}
\newtheorem{definition}[theorem]{Definition}

\newcommand{\bbE}{{\ensuremath{\mathbb E}} }

\newcommand{\bbN}{{\ensuremath{\mathbb N}} }

\newcommand{\bbR}{{\ensuremath{\mathbb R}} }

\newcommand{\cI}{{\ensuremath{\mathcal I}} }

\newcommand{\cZ}{{\ensuremath{\mathcal Z}} }

\newcommand{\gb}{\beta}

\newcommand{\gt}{\vartheta}

\newcommand{\gl}{\lambda}

\newcommand{\gO}{\Omega}

\DeclareMathSymbol{\leqslant}{\mathalpha}{AMSa}{"36} 
\DeclareMathSymbol{\geqslant}{\mathalpha}{AMSa}{"3E} 
\DeclareMathSymbol{\eset}{\mathalpha}{AMSb}{"3F}     

\newcommand{\be}{\begin{equation}}
	\newcommand{\ee}{\end{equation}}

\newcommand{\R}{\mathbb{R}}

\newcommand{\PEfont}{\mathrm}
\DeclareMathOperator{\var}{\ensuremath{\PEfont Var}}
\DeclareMathOperator{\cov}{\ensuremath{\PEfont Cov}}

\newcommand{\e}{\ensuremath{\PEfont E}}

\newcommand{\E}{\e}

\newcommand{\ind}{\mathds{1}}

\newcommand{\eps}{\varepsilon}
\renewcommand{\epsilon}{\varepsilon}
\renewcommand{\theta}{\vartheta}
\renewcommand{\rho}{\varrho}

\newenvironment{myenumerate}{%
	\renewcommand{\theenumi}{\arabic{enumi}}%
	\renewcommand{\labelenumi}{{\rm(\theenumi)}}%
	\begin{list}{\labelenumi}
		{%
			\setlength{\itemsep}{0.4em}%
			\setlength{\topsep}{0.5em}%
			\setlength\leftmargin{2.45em}%
			\setlength\labelwidth{2.05em}%
			\setlength{\labelsep}{0.4em}%
			\usecounter{enumi}%
		}%
	}%
	{\end{list}
}

	{\end{list}
}

	{\end{list}
}

\renewenvironment{enumerate}{
	\begin{myenumerate}}%
	{\end{myenumerate}}

\newenvironment{myitemize}{%
	\begin{list}{$\bullet$}%
		{%
			\setlength{\itemsep}{0.4em}%
			\setlength{\topsep}{0.5em}%
			\setlength\leftmargin{2.65em}%
			\setlength\labelwidth{2.65em}%
			\setlength{\labelsep}{0.4em}%
		}%
	}%
	{\end{list}}

	{\end{myitemize}}

\MakePerPage[2]{footnote} 

\date{\today}

\newcommand\dd{\mathrm{d}}

\newcommand\sZ{\mathscr Z}

\newcommand\td{\text{d} }

\usepackage{mathtools}

\title[Moments comparison inequality for SHF, DPRE and GMCs] {Moments Comparison Inequalities for Critical 2d Stochastic Heat Flow, Polymers and GMC}
	
\begin{document}

		\author{Ziyang Liu, Zuodi Xie}
		\email{ziyang.liu.1@warwick.ac.uk, zuodi-xie.xie@warwick.ac.uk}
		\begin{abstract}
			We employ a Gaussian convex inequality developed  in \cite{AC_2015_convex_positive_lambda},  \cite{Chen2026_Ehrhard_convex1}, \cite{Guerra_2022_convex_positive_lambda} and \cite{F2025_Gaussian_convex_inequality} to establish  (one sided) moment comparison inequalities for Gaussian multiplicative chaos, directed polymers in Gaussian environment and, in particular, the Critical 2d Stochastic Heat Flow (SHF), including upper bounds on negative moments, as well as similar lower bounds for $p$-th moments with $p\in (0,1)$.  For the Critical 2d SHF, averaged over small balls, we also obtain the matching (up to multiplicative constants)  complementary bounds.  
			We establish the complementary bounds using
			a  bootstrap argument inspired by \cite{DS2010_liouville_and_KPZ} and a geometric decomposition introduced in \cite{GuTsai_2026_loglogCLT}. Specifically, we prove that for locally averaged SHF, for every $p\in\mathbb R$, the $p$-th moment is, up to multiplicative constants, bounded in both sides by the $\frac{p(p-1)}{2}$-th power of its second moment.

		\end{abstract}
		
		\keywords{Gaussian convexity inequality, critical stochastic heat flow, Ehrhard-type random variables, Gaussian multiplicative chaos, directed polymers, local moments}
		
		\subjclass[2020]{Primary 60E15, 60H15; Secondary 60G57, 60K37, 82D60}

		\maketitle
		
		\tableofcontents

		\section{Introduction}
		
		\subsection{Background and motivations}
		\label{sec:critical-SHF}
		
		The Critical 2D Stochastic Heat Flow (SHF) was constructed in \cite{CSZ2023_2dcSHF} as a natural candidate for the solution of the 2D multiplicative stochastic heat equation at criticality, and also as a non-trivial, non-Gaussian limit of the rescaled partition function of the critical two-dimensional Directed Polymer in Random Environment (DPRE). We refer to \cite{CSZ2024_reviewSHF,CSZ2025_SHFicm} for reviews of SHF and to \cite{Zygouras2024_DPREreview,Comets2017_DPREbook} for reviews of DPRE.
		
		Integer moments of the Critical 2D SHF play an important role in the study of its probabilistic properties. Some moment-convergence results for the mollified stochastic heat equations and DPRE can be found in \cite{BC1998_secondmomentSHE,CSZ2019_Dickman,CSZ2019_thirdmoment,GQT_2021_moments, CoscoZeitouni2023_momentsI,CoscoZeitouni2024_momentsII,CoscoNakajima2025_highmomentsDPRE}. Integer-moment estimates also reveal important features such as the intermittency of SHF at both macroscopic and microscopic scales. Nevertheless, a fuller understanding of SHF requires information beyond its positive integer moments. Recent developments on fractional moments of SHF can be found in \cite{BergerTurchiZygouras2026_fractionalSHF,Huang2026_fractionalSHF}.
		
		Recent integer-moment estimates reveal a strong intermittency phenomenon for SHF at the macroscopic scale. In \cite{GangulyNam2025_SHFmoments_lowerbound}, it is proved that the $h$-th moment of SHF tested on a macroscopic ball grows doubly exponentially as $h\to\infty$. Since the first moment of SHF is finite, this doubly exponential growth shows that high moments are dominated by rare events carrying exceptionally large masses. In other words, positive moments and upper-tail estimates, which describe sparse high peaks, capture the macroscopic intermittency of SHF. By contrast, the probability of SHF taking value close to zero, and in particular its negative moments, is much less understood; see Open Problem~5 in \cite{CSZ2024_reviewSHF}. Recently, \cite{Nakashima_disjoint_tube} established a left-tail estimate for the logarithm of SHF, but this estimate does not imply the finiteness of any negative moment of SHF.
		
		In the subcritical regime, \cite{CSZ2020_subcriticalKPZ} established lower-tail estimates for the two-dimensional DPRE and used them to control negative moments of its partition function. These estimates also play an important role in the construction of the two-dimensional KPZ equation in the subcritical regime \cite{CSZ2020_subcriticalKPZ}. The argument there relies on a localized Gaussian concentration inequality together with control of the Lipschitz constant of the free energy. At criticality, however, this Lipschitz constant grows too rapidly for the same method to be applied directly. By contrast, the Gaussian convexity property of \cite{Chen2026_Ehrhard_convex1} yields a Gaussian left-tail bound for convex functionals of Gaussian vectors in terms of the variance of the random variable, rather than the Lipschitz constant of the convex function. This provides the crucial mechanism for our analysis of negative moments and lower tails of the two-dimensional critical DPRE model and its limiting SHF. A more detailed discussion on this is given in Subsection~\ref{subsec: discussion on the proofs}.
		
		Another intermittency phenomenon of SHF appears in its microscopic behavior. It was proved in \cite{LZ2024_micromomentSHF} that the locally averaged SHF has positive integer moments reminiscent of those of log-normal random variables. Although SHF itself is not a Gaussian multiplicative chaos, see \cite{CSZ2023_SHFisnotGMC}, this local-moment behavior indicates that its microscopic fluctuations exhibit a log-normal structure. This picture is further supported by \cite{GuTsai_2026_loglogCLT}, which proves a central limit theorem for the renormalized logarithm of the locally averaged SHF. However, a CLT for the locally averaged SHF cannot by itself yield sharp local-moment estimates. Moreover, the result of \cite{LZ2024_micromomentSHF} applies only to positive integer moments of SHF. The method there, which relies on a collision-diagram expansion derived from the integer-moment formula for SHF, cannot be directly extended to fractional or negative moments. It is also unclear from \cite{LZ2024_micromomentSHF} whether the $o(1)$ term in the exponent of the upper bound is intrinsic or merely an artifact of the proof.
		
		Recently, \cite{BergerTurchiZygouras2026_fractionalSHF} proved an upper bound on the $p$-th fractional moments of SHF, uniformly in $p\in(0,1)$, the time parameter, the disorder parameter and the averaging radius within square root of the time parameter. In \cite{Huang2026_fractionalSHF}, a local fractional-moment estimate was also announced, although its logarithmic exponent still contains an additional $o(1)$ term.
		
		Beyond their intrinsic interest, sharp local-moment estimates are also important inputs for the study of the logarithmic multifractal structure of SHF. Roughly speaking, multifractal analysis of an irregular random measure aims to quantify the size of its high-value points; see \cite{Bertacco_2023_MultifractalGMC} and the references therein. For SHF, we are interested in the size of level sets consisting of points at which the local average grows at a prescribed logarithmic rate. Closely related to this question, the recent work \cite{GN2026_coveringnumberSHF} studies the covering number by small balls of a carrier of SHF. Our sharp moment estimate for the locally averaged SHF provides another way to prove the result therein (see Remark \ref{remark: on GN covering number}). We expect that sharp local-moment estimates will also be a fundamental ingredient in the multifractal analysis, which we plan to explore in future work.
		
		Motivated by the significance of moment estimates for SHF discussed above, this paper makes two principal contributions. First, we establish a moment comparison inequality for SHF. Besides solving two open questions, including the finiteness of all negative moments and estimates for the lower bound of $p\in (0,1)$ fractional moments, this inequality provides a general mechanism for deriving new moment estimates from a limited collection of known bounds. Second, we determine sharp two-sided asymptotics for the moments of locally averaged SHF of every real order, thereby extending the results of \cite{LZ2024_micromomentSHF}. 
		
		Although the main purpose of the present paper is to establish moment estimates for the critical SHF, a broader goal is to introduce a general framework for moment estimates in Gaussian models and their scaling limits. Within the setting of SHF and DPRE, we exploit the Gaussian convexity property developed in \cite{AC_2015_convex_positive_lambda},  \cite{Chen2026_Ehrhard_convex1}, \cite{Guerra_2022_convex_positive_lambda} and \cite{F2025_Gaussian_convex_inequality} as a mechanism for deriving previously inaccessible moment estimates from known moment bounds. We also apply the same argument to subcritical Gaussian multiplicative chaos; see \cite{RV_2014_GMCreview} for a review, and to mollified stochastic heat equations (mSHE). Although the Gaussian convexity property has already been used in the study of spin glasses, for example \cite{Chen2026_Ehrhard_convex1}, to the best of our knowledge it has not yet been systematically applied to directed polymers, GMC, or mSHE. We expect that the method and framework developed here will also apply to other random models admitting suitable convex Gaussian approximations.
		
		\subsection{Main results}
		\label{subsection: main result}
		{
			The main results of this work contain two parts:
			\begin{enumerate}
				\item Applications of the Gaussian convexity property in \cite{AC_2015_convex_positive_lambda},  \cite{Chen2026_Ehrhard_convex1}, \cite{Guerra_2022_convex_positive_lambda} and \cite{F2025_Gaussian_convex_inequality} to GMC, DPRE, and SHF.  
				\item A sharp two-sided estimate for the $p$-th moments of locally averaged SHF for every $p\in \mathbb R$. 
			\end{enumerate} 
		}
		
		\subsubsection{Applications of Gaussian convexity property}
		We first introduce the Gaussian convexity property, which provides a common framework for the moment estimates developed in this paper for Gaussian DPRE, SHF, and GMC. We exploit a common structure of these models: their logarithms can be represented as \emph{convex functionals of Gaussian vectors}, or as weak limits of such functionals. By Ehrhard's inequality, see \cite{Ehrhard1983,Borell2003_Ehrhard}, if $X=F(G)$, where $F$ is convex and $G$ is a finite-dimensional Gaussian vector, then
		\begin{equation}\label{eqIntro: concavity}
			r\mapsto \Phi^{-1} \bigl(\mathbb P(X<r)\bigr)
			\quad\text{is concave},
		\end{equation}
		where $\Phi(a):=\mathbb P(\mathcal N(0,1)<a)$ denotes the cumulative distribution function of a standard Gaussian random variable. Motivated by Ehrhard's inequality, we call random variables satisfying this concavity property \emph{Ehrhard-type}; a precise definition is given in Definition~\ref{def: Ehrhard-type random variables}.
		In particular, the logarithms of DPRE partition functions are Ehrhard-type because they are convex functionals of Gaussian vectors. In Lemma~\ref{lemma: Ehrhard type is close}, we show that the Ehrhard-type property is closed under weak convergence. As a result, the logarithms of subcritical GMC and SHF masses are also Ehrhard-type.  
		
		In \cite{AC_2015_convex_positive_lambda}, \cite{Chen2026_Ehrhard_convex1} and \cite{Guerra_2022_convex_positive_lambda}, it is proved that the normalized log-moment generating function; see \eqref{eq: normalized log MGF}, of a convex functional of some Gaussian vector is convex. Then, it is found in \cite{F2025_Gaussian_convex_inequality} that such convexity can be extended for Ehrhard-type random variables which are essentially still convex-Gaussian under the meaning of distribution (see Lemma \ref{lemma: Ehrhard and convex Gaussian}). As a direct corollary, we obtain the following moment comparison inequalities:
		\begin{theorem}\label{thm: intro, Ehrhard moment comparison ineq}
			Let $X$ be an Ehrhard-type real-valued random variable in the sense of Definition~\ref{def: Ehrhard-type random variables}, and set $M:=e^X$. Assume that $\mathbb E[M]=1$. Then
			\begin{equation}\label{eqMC: intro, moment comparison ineq}
				\frac{\log \mathbb E[M^p]}{p(p-1)}
				\le
				\frac{\log \mathbb E[M^q]}{q(q-1)},
				\qquad
				\forall\, p<q,\quad p,q\notin\{0,1\},
			\end{equation}
			and
			\begin{equation}\label{eqMC: intro, E X two-side}
				-\frac{\log \mathbb E[M^q]}{q(q-1)}
				\le
				\mathbb E[X]
				\le
				-\frac{\log \mathbb E[M^p]}{p(p-1)},
				\qquad
				\forall\, p<0,\quad q>0.
			\end{equation}
			In particular,
			\begin{enumerate}
				\item if $p\in(-\infty,0)\cup(1,2)$, then
				$
				\mathbb E[M^p]
				\le
				\left(\mathbb E[M^2]\right)^{p(p-1)/2};
				$
				\item if $p\in(0,1)$, then
				$
				\mathbb E[M^p]
				\ge
				\left(\mathbb E[M^2]\right)^{p(p-1)/2};
				$
				\item if $q>2$, then
				$
				\mathbb E[M^q]
				\ge
				\left(\mathbb E[M^2]\right)^{q(q-1)/2}.
				$
			\end{enumerate}
		\end{theorem}
		
		As a first application, we consider subcritical Gaussian multiplicative chaos which provides a rigorous construction of the exponential of a log-correlated Gaussian field. We refer to \cite{RV_2014_GMCreview} for a review. In Subsection~\ref{subsection: GMC}, we show that the logarithm of the mass of a subcritical GMC is Ehrhard-type and hence obtain bounds for its negative moments from Theorem~\ref{thm: intro, Ehrhard moment comparison ineq}. We postpone further details to Subsection~\ref{subsection: GMC}.
		Another closely related model considered here is the solution of the 2D mollified stochastic heat equation. In Subsection \ref{subsection: mSHE}, we prove that the logarithm of the solution is also Ehrhard-type and hence satisfies the moment comparison inequality in Theorem~\ref{thm: intro, Ehrhard moment comparison ineq}.
		
		In what follows, we introduce the applications of Theorem~\ref{thm: intro, Ehrhard moment comparison ineq} in DPRE and SHF. We start with the definitions of partition functions of DPRE models. Let $S=(S_n)_{n\ge0}$ be a simple symmetric random walk on $\mathbb Z^2$, with
		$
		q_n(z):=\mathrm P(S_n=z\mid S_0=0),
		$
		and let
		$
		\bigl\{\omega(n,z): n\in\mathbb Z,\ z\in\mathbb Z^2\bigr\}
		$
		be a family of i.i.d. standard Gaussian random variables, independent of $S$. For integers $m<n$ and $x,y\in\mathbb Z^2$, define the point-to-point partition function by
		\begin{equation}\label{eqDPRE: Z_{m,n}^{beta,omega}}
			\mathcal Z_{m,n}^{\beta,\omega}(x,y)
			:=
			\mathrm E\!\left[
			\exp\!\left\{
			\sum_{k=m+1}^{n-1}
			\left(
			\beta\omega(k,S_k)-\frac{\beta^2}{2}
			\right)
			\right\}
			\mathbf 1_{\{S_n=y\}}
			\,\middle|\, S_m=x
			\right],
		\end{equation}
		where the expectation is taken with respect to the random walk $S$. In Lemma~\ref{lemma: convexity of Free energy}, we show that for any two functions $f_1,f_2:\mathbb Z^2\to\mathbb R_+$, with at least one of them compactly supported, the logarithm of the partition function
		\[
		\log \mathcal Z_{m,n}^{\beta,\omega}(f_1,f_2)
		:=
		\log\!\left(
		\sum_{x,y\in\mathbb Z^2}
		\mathcal Z_{m,n}^{\beta,\omega}(x,y)f_1(x)f_2(y)
		\right)
		\]
		is a convex functional of a finite-dimensional Gaussian vector and is therefore Ehrhard-type. As a direct application of Theorem~\ref{thm: intro, Ehrhard moment comparison ineq}, we obtain the following.
		
		\begin{theorem}\label{thm: DPRE, moment comparison}
			The renormalized partition function $\bar{\mathcal Z}_{m,n} ^{\beta,\omega}(f_1,f_2):= \frac{ 
				\mathcal Z_{m,n}^{\beta,\omega}(f_1,f_2)}
			{\mathbb E [\mathcal Z_{m,n}^{\beta,\omega}(f_1,f_2)]}$ satisfies the Ehrhard-type inequalities \eqref{eqMC: intro, moment comparison ineq}--\eqref{eqMC: intro, E X two-side} with
			$
			M:=\bar {\mathcal Z}_{m,n}^{\beta,\omega}(f_1,f_2),\ 
			X:=\log M.
			$
		\end{theorem}
		
		\begin{remark}
			Theorem~1.5 of \cite{BergerNakajima2026_sharpfreeenergy} gives two-sided estimates for
			\[
			F_p(\beta)
			:=\lim_{N\to \infty }
			\frac{1}{N}\log\mathbb E\!\left[(Z_N^{\beta,\omega})^p\right],
			\qquad
			Z_N^{\beta,\omega}
			:=
			\sum_{z\in\mathbb Z^2}Z_{0,N}^{\beta,\omega}(0,z),
			\]
			for $p>0$, under suitable moment conditions on $\omega$. By Theorem 1.3 and Theorem 1.5 therein, there is $c$, $C>0$ such that, as $\beta\to 0$,
			\[
			\begin{aligned}
				p\in (0,1),\quad 
				-c p\wedge (1-p)F_2(\beta) \le F_p(\beta ) \le -C p\wedge (1-p)F_2(\beta),\\
				p>1,\quad 
				C (p-1)F_2(\beta) \le F_p(\beta ) \le c (p-1)F_{\lceil p\rceil}(\beta).
			\end{aligned}
			\]
			By Theorem~\ref{thm: DPRE, moment comparison}, when $\omega$ is Gaussian we obtain a one-sided bound for $F_p(\beta)$: if $p\in(-\infty,0)\cup(1,2)$, then
			\[
			F_p(\beta)\le \frac{p(p-1)}{2}F_2(\beta),
			\]
			whereas if $p\in(0,1)\cup(2,\infty)$, then
			\[
			F_p(\beta)\ge \frac{p(p-1)}{2}F_2(\beta).
			\]
			One cannot, however, expect to obtain bounds in the opposite direction directly from the Gaussian convexity property. It would also be interesting to extend the present results to DPRE with non-Gaussian disorder.
		\end{remark}
		
		\begin{remark}
			Consider the point-to-plane partition function of the two-dimensional subcritical DPRE, i.e.
			$
			W_N:= \sum_{z\in \mathbb Z^2} Z_{0,N}^{\beta_N, \omega} (0,z),
			$
			with $\beta_N^2 = \hat \beta ^2 \frac{\pi}{\log N}$ for some $\hat \beta \in (0,1)$. \cite{CoscoZeitouni2024_momentsII} studies the lower bound of integer high moments:
			For an integer $q$ in the regime $O(\log \log N)\le q^2\le O(\log N)$, their theorem 1.1 gives 
			\begin{equation}\label{eq: remark, CZ, II, thm1}
				\mathbb E [W_N^q] \ge e^{\lambda^2 (1-o_N(1) ) \frac{q(q-1)}{2}},\quad \lambda^2 = \log \frac{1}{1-\hat \beta ^2}. 
			\end{equation}
			By Theorem \ref{thm: DPRE, moment comparison} and Theorem \ref{thm: intro, Ehrhard moment comparison ineq}, we have 
			\begin{equation}\label{eq: remark on CZ II}
				\mathbb E [W_N^q] \ge (\mathbb E [W_N^2] )^{\frac{q(q-1)}{2}}
			\end{equation}
			for any real $q\ge 2$. From $\lim _N \mathbb E [W_N^2]=e ^{\lambda ^2}$ (refer to \cite{LygkonisZygouras2023_ErdosTaylor}), we can thus recover the lower bound in \eqref{eq: remark, CZ, II, thm1}. However, when $q>C \log N$ for some $C>0$, theorem 1.3 in \cite{CoscoZeitouni2024_momentsII} gives $\mathbb E [W_N^q] \ge e ^{c \frac{q(q-1)}{2} \frac{N}{\log N} }$ which is much larger than the lower bound given in \eqref{eq: remark on CZ II}. Essentially, the exploding term $\frac{N}{\log N}$ appearing in the exponent comes from the fact that the contribution generated by the intersection of $q$ independent random walks dominates the probability cost of such collision event when $q=O(\log N)$.
		\end{remark}

			We now turn to the critical two-dimensional DPRE model and its weak limit, the Critical 2D SHF. Fix $d=2$ and write
			\[
			R_N:=\sum_{n=1}^N\sum_{z\in\mathbb Z^2}q_n(z)^2
			=\sum_{n=1}^N q_{2n}(0)
			=\frac{\log N}{\pi}+O(1).
			\]
			We choose the Gaussian disorder strength $\beta_N=\beta_N(\theta)$ through the critical-window condition
			\begin{equation}\label{eqDPRE: beta_N, critical}
				\begin{aligned}
					\sigma_N^2
					&:=\mathbb E\!\left[\left(e^{\beta_N\omega-\beta_N^2/2}-1\right)^2\right]
					=e^{\beta_N^2}-1\\
					&=\frac{1}{R_{2N} }
					\left(1+\frac{\theta+\gamma_{\mathrm E}+o(1)}{\log (2N )}\right),
					\qquad \theta\in\mathbb R,
				\end{aligned}
			\end{equation}
			where $\gamma_{\mathrm E}:=-\int_0^\infty e^{-u}\log u\,\td u$ is the Euler--Mascheroni constant.
		In contrast, the DPRE model with
		\begin{equation}\label{eqDPRE: beta_N, subcritical}
			\beta_N^2
			:=
			\frac{\pi\widehat\beta^2}{\log N},
			\qquad \widehat\beta\in(0,1),
		\end{equation}
		is said to lie in the subcritical regime; see \cite{CSZ2024_reviewSHF} and the references therein.

		For $s<t$, define the diffusively rescaled random endpoint measure by
		\begin{equation}\label{eq:rescaled-partition-measure}
			\mathcal Z_{N;s,t}^{\theta}(\td x,\td y)
			:=\frac N2\,
			\mathcal Z_{\llbracket 2Ns\rrbracket,\llbracket 2Nt\rrbracket}^{\beta_N,\omega}
			\bigl(\llbracket\sqrt N\,x\rrbracket,
			\llbracket\sqrt N\,y\rrbracket\bigr)\,\td x\,\td y.
		\end{equation}
		For a real number $T$, set $\llbracket T\rrbracket:=2\lfloor T/2\rfloor$. For $w\in\mathbb R^2$, let $\llbracket w\rrbracket$ denote its nearest point in
		\[
		\mathbb Z^2_{\mathrm{even}}
		:=\{z=(z_1,z_2)\in\mathbb Z^2:z_1+z_2\in2\mathbb Z\}.
		\]

		For every finite collection
		$
		s_1<t_1,\ldots,s_k<t_k,
		$
		the random measures in \eqref{eq:rescaled-partition-measure} converge jointly, in the vague topology, to a unique-in-law, non-degenerate family of locally finite positive random measures \cite{CSZ2023_2dcSHF}:
		\begin{equation}\nonumber
			\left(
			\mathcal Z_{N;s_i,t_i}^{\theta}
			\right)_{i=1}^{k}
			\Longrightarrow
			\left(
			\mathscr Z_{s_i,t_i}^{\theta}
			\right)_{i=1}^{k}
			\qquad\text{as }N\to\infty.
		\end{equation}
		The limiting family
		$
		\mathscr Z^{\theta}
		:=
		\bigl\{\mathscr Z_{s,t}^{\theta}:-\infty<s<t<\infty\bigr\}
		$
		is called the \emph{Critical Two-dimensional Stochastic Heat Flow}.
		
			\begin{remark}\label{rem:SHF-normalization-CSZ}
				Only in this remark, write $\mathscr Z^{\vartheta,\mathrm{CSZ}}$ for the original construction in \cite{CSZ2023_2dcSHF}, which is defined as limit of 
				\begin{align*}
					\frac N4\,
					\mathcal Z_{\llbracket Ns\rrbracket,\llbracket Nt\rrbracket}^{\bar \beta_N,\omega}
					\bigl(\llbracket\sqrt N\,x\rrbracket,
					\llbracket\sqrt N\,y\rrbracket\bigr)\,\td x\,\td y
				\end{align*}
				where 
				\begin{align*}
					e^{\bar \beta_N^2}-1
					=
					\frac{1}{R_N} \Big(1+\frac{\vartheta+o(1)}{\log N}\Big).
				\end{align*}
				This contains a rescale on \eqref{eq:rescaled-partition-measure} and a constant shift on \eqref{eqDPRE: beta_N, critical}. Denote its weak limit as $\mathscr Z_{s,t}^{\theta ,\text{CSZ}} (\dd x, \dd y)$.
				Then we have the rescaling formula:
				\begin{equation}\label{eq:SHF-CM-CSZ-conversion}
					\mathscr Z_{s,t}^{\theta}(\td x,\td y)
					\overset{d}{=}
					2\mathscr Z_{2s,2t}^{\theta+\gamma_{\mathrm E}-\log2,\mathrm{CSZ}}
					(\td x,\td y),
				\end{equation}
				jointly over finitely many time intervals. In particular,
				\[
				\mathbb E\mathscr Z_{s,t}^{\theta}(\td x,\td y)
				=g_{t-s}(y-x)\,\td x\,\td y,
				\]
				while $\mathbb E\mathscr Z_{s,t}^{\theta,\text{CSZ}}(\td x,\td y)
				= \tfrac12g_{(t-s)/2}(y-x)\,\td x\,\td y$. This is precisely the normalization of \cite[Section~1]{ClarkMian2026_SHFflow} and \cite[Definition~1.2 and Section~2.1]{Tsai2024_SHF_by_moments}.
			\end{remark}
		
			For any non-negative $f\in L^1(\mathbb R^2)$, write
			\[
			\mathscr Z_t^\theta(f)
			:=\mathscr Z_{0,t}^\theta(f, \mathbf1)
			=\int_{\mathbb R^4}f(x)\,\mathscr Z_{0,t}^\theta(\td x,\td y).
			\]
			On the discrete model, we first define the cells 
			\[
			C_N(z):=\{x\in\mathbb R^2:\llbracket\sqrt N\,x\rrbracket=z\},
			\qquad z\in\mathbb Z^2_{\mathrm{even}},
			\]
			partition $\mathbb R^2$ up to their boundaries and have Lebesgue area $2/N$.
			Then we define the cell averages:
			\begin{equation}\label{eqDPRE: cell-averaged test}
				f_N(z):=\frac N2\int_{C_N(z)}f(x)\,\td x,
				\qquad
				\frac2N\sum_{z\in\mathbb Z^2_{\mathrm{even}}}f_N(z)
				=\int_{\mathbb R^2}f(x)\,\td x.
			\end{equation}
			The marginal DPRE partition function is then given by
			\begin{equation}\label{eqDPRE: Z_N(f)}
				\begin{aligned}
					\mathcal Z_{\llbracket2Nt\rrbracket}^{\beta_N(\theta)}(f)
					&:=\frac2N\sum_{z\in\mathbb Z^2_{\mathrm{even}}}
					f_N(z)\,\mathcal Z_{0,\llbracket2Nt\rrbracket}^{\beta_N,\omega}(z,\mathbf1)\\
					&=\frac2N\sum_{z\in\mathbb Z^2_{\mathrm{even}}}f_N(z)\,
					\mathrm E_z\!\left[
					\exp\!\left\{\sum_{n=1}^{\llbracket2Nt\rrbracket-1}
					\left(\beta_N(\theta)\omega(n,S_n)
					-\frac{\beta_N(\theta)^2}{2}\right)\right\}\right].
				\end{aligned}
			\end{equation}
			Here $\mathcal Z_{0,n}^{\beta,\omega}(z,\mathbf1):=\sum_y\mathcal Z_{0,n}^{\beta,\omega}(z,y).$ Then $	\mathcal Z_{\llbracket2Nt\rrbracket}^{\beta_N(\theta)}(f)$ weakly converges to $\mathscr Z_{t}^{\theta } (f,\mathbf 1) =: \mathscr Z_{t} ^{\theta} (f)$.

		In the discussion below, we will use the two facts on the SHF:
		\[
		\mathbb E[\mathscr Z_t^\theta(f)]=\|f\|_{L^1(\mathbb R^2)},
		\qquad
		\mathbb E[\mathscr Z_t^\theta(f)^2]<\infty
		\quad\text{for }0\le f\in L^1(\mathbb R^2)\cap L^\infty(\mathbb R^2).
		\]
		Our first main result on the SHF is the following moment comparison inequality for SHF which follows from Theorem \ref{thm: intro, Ehrhard moment comparison ineq}. 
		
		\begin{theorem}\label{thm: SHF moment comparison ineq}
			Fix $t>0$ and $\theta\in\mathbb R$. Let $\mathscr Z_t^\theta(\td x)$ be the two-dimensional critical SHF. Let either $f\in C_c(\mathbb R^2)$ or $f=c\mathbf 1_A$ for some compact set $A\subset\mathbb R^2$ and $c>0$. Assume further that $f\ge0$ and $\|f\|_{L^1(\mathbb R^2)}=1$. Then, for $p\in(-\infty,0)\cup(1,2)$,
			\begin{equation}\label{eqSHF: p<0, 1<p<2 moments upper bounds}
				\mathbb E[\mathscr Z_t^\theta(f)^p]
				\le
				\left(\mathbb E[\mathscr Z_t^\theta(f)^2]\right)^{p(p-1)/2}
				<\infty.
			\end{equation}
			More generally,
			\begin{equation}\label{eqSHF: moment comparison inequality}
				\frac{\log\mathbb E[\mathscr Z_t^\theta(f)^p]}{p(p-1)}
				\le
				\frac{\log\mathbb E[\mathscr Z_t^\theta(f)^q]}{q(q-1)},
				\qquad
				\forall\,p<q,\quad p,q\notin\{0,1\},
			\end{equation}
			and
			\begin{equation}\label{eqSHF: E log Z two-side bound}
				-\frac{\log\mathbb E[\mathscr Z_t^\theta(f)^q]}{q(q-1)}
				\le
				\mathbb E[\log\mathscr Z_t^\theta(f)]
				\le
				-\frac{\log\mathbb E[\mathscr Z_t^\theta(f)^p]}{p(p-1)},
				\qquad
				\forall\,p<0,\quad q>0.
			\end{equation}
		\end{theorem}
		
		We discuss here the results in Theorem~\ref{thm: SHF moment comparison ineq}:
		\begin{enumerate}
			\item Taking $p<0$ in \eqref{eqSHF: p<0, 1<p<2 moments upper bounds}, we obtain the finiteness of every negative moment of SHF and therefore solve Open Problem~5 in \cite{CSZ2024_reviewSHF}. The existence of negative moments also implies that SHF is positive almost surely, so it is legitimate to write $\log\mathscr Z_t^\theta(f)$ in \eqref{eqSHF: E log Z two-side bound}. Indeed, using $|\log z|\le C(z+z^{-1})$ for $z>0$, we obtain
			\begin{equation}\label{eqSHF: E log Z <}
				\begin{aligned}
					\left|\mathbb E[\log\mathscr Z_t^\theta(f)]\right|
					&\le
					C\left(
					\mathbb E[\mathscr Z_t^\theta(f)]
					+\mathbb E[\mathscr Z_t^\theta(f)^{-1}]
					\right)\\
					&\le
					C\left(
					\mathbb E[\mathscr Z_t^\theta(f)]
					+\mathbb E[\mathscr Z_t^\theta(f)^2]
					\right)
					<\infty,
				\end{aligned}
			\end{equation}
			where the second inequality follows from \eqref{eqSHF: p<0, 1<p<2 moments upper bounds} with $p=-1$. Positivity of SHF has also been proved in \cite{Nakashima_disjoint_tube,ClarkTsai_2025_cGMC} by different methods. Here we reprove it directly from the existence of negative moments.
			
			\item Taking $p\in(0,1)$ and $q=2$ in \eqref{eqSHF: moment comparison inequality}, we obtain
			\[
			\mathbb E[\mathscr Z_t^\theta(f)^p]
			\ge
			\left(
			\mathbb E[\mathscr Z_t^\theta(f)^2]
			\right)^{p(p-1)/2},
			\qquad p\in(0,1),
			\]
			which proves the lower-bound part of Theorem~1.1 in \cite{BergerTurchiZygouras2026_fractionalSHF} for the fractional moments of SHF.
		\end{enumerate}
		
		Our second main result concerns the left tail of SHF. In \cite{Nakashima_disjoint_tube}, it is proved that for every $\delta<2/3$,
		\begin{equation}\nonumber
			\mathbb P\!\left(
			\log\mathscr Z_t^\theta(B(0,1))<-r
			\right)
			\le e^{-cr^\delta}
		\end{equation}
		for all sufficiently large $r>0$. We improve this to a Gaussian left-tail bound, which follows essentially from the left-tail estimate of \cite{Chen2026_Ehrhard_convex1}; see Theorem~\ref{thm: Chen, theorem 3, left tail}.
		
		\begin{theorem}[Gaussian left tail of SHF]\label{thm: left tail of SHF}
			For any $t>0$ and $\theta\in\mathbb R$, and for any non-negative $f\in C_c(\mathbb R^2)$ or $f=c\mathbf 1_A$ with $c>0$ and $A\subset\mathbb R^2$ compact, we have
			\begin{equation}\label{eqSHF: left tail}
				\mathbb P\!\left(
				\log\mathscr Z_t^\theta(f)
				-\mathbb E[\log\mathscr Z_t^\theta(f)]
				<-x\sqrt{\operatorname{Var}(\log\mathscr Z_t^\theta(f))}
				\right)
				\le
				e^{-x^2/2},
				\qquad x>0,
			\end{equation}
			where $\mathbb E[\log\mathscr Z_t^\theta(f)]$ can be bounded using \eqref{eqSHF: E log Z <}.
		\end{theorem}
		
		\subsubsection{Moments estimate of locally averaged SHF}
		Set
		\begin{equation}\label{eqIntro: U epsilon}
			\mathcal U_\epsilon(x)
			:=
			\frac{1}{\pi\epsilon^2}\mathbf 1_{B(0,\epsilon)}(x),
			\qquad
			B(0,\epsilon)
			:=
			\{z\in\mathbb R^2:|z|\le\epsilon\},
		\end{equation}
		where $|\cdot|$ denotes the Euclidean norm on $\mathbb R^2$. It is proved in \cite{LZ2024_micromomentSHF} that for every integer $p\ge2$,
		\begin{equation}\label{eqSHF: LZ local moments}
			c_p \left(
			\mathbb E[\mathscr Z_t^\theta(\mathcal U_\epsilon)^2]
			\right)^{p(p-1)/2}
			\le
			\mathbb E[\mathscr Z_t^\theta(\mathcal U_\epsilon)^p]
			\le
			C_p
			\left(
			\mathbb E[\mathscr Z_t^\theta(\mathcal U_\epsilon)^2]
			\right)^{p(p-1)/2+C/\log^{(3)}(1/\epsilon)},
		\end{equation}
		with
		$
		\mathbb E[\mathscr Z_t^\theta(\mathcal U_\epsilon)^2]
		\asymp
		\log\frac{1}{\epsilon};
		$
		see \eqref{eqSHF: local second moments}. We extend this local-moment estimate in two directions: we remove the $\log^{(3)}(1/\epsilon)$ term from the exponent in the upper bound in \eqref{eqSHF: LZ local moments}, and we extend the estimates to all real moments.
		
		\begin{theorem}[Moments estimate of locally averaged SHF]\label{thm: SHF, local moments}
			Let $\mathcal U_\epsilon$ be defined by \eqref{eqIntro: U epsilon}. For every $p\in\mathbb R$, there exist constants
			$
			C_L=C_L(\theta,t,p)\in (0,1)$, $
			C_U=C_U(\theta,t,p)\ge1,
			$
			and $\epsilon_0>0$ such that for every $\epsilon\in(0,\epsilon_0)$,
			\begin{equation}\label{eqSHF: local moments estimate}
				C_L
				\left(
				\mathbb E[\mathscr Z_t^\theta(\mathcal U_\epsilon)^2]
				\right)^{p(p-1)/2}
				\le
				\mathbb E[\mathscr Z_t^\theta(\mathcal U_\epsilon)^p]
				\le
				C_U
				\left(
				\mathbb E[\mathscr Z_t^\theta(\mathcal U_\epsilon)^2]
				\right)^{p(p-1)/2}.
			\end{equation}
			In particular, one may take $C_L=1$ for $p\in[2,\infty)\cup[0,1]$ and $C_U=1$ for $p\in(-\infty,0)\cup(1,2)$.
		\end{theorem}
		
		\begin{remark}
			Compared with \cite{BergerTurchiZygouras2026_fractionalSHF}, Theorem~\ref{thm: SHF, local moments} gives $p$-moment estimates for both sides and all $p\in\mathbb R$ but only for SHF tested on the microscopic scale. In contrast, the upper bound obtained in \cite{BergerTurchiZygouras2026_fractionalSHF} is uniform in $p\in(0,1)$, $\theta\in\mathbb R$, $t>0$, and $r\le\sqrt t$. This uniformity makes it possible to study regimes in which the test region, for example
			$
			\{z\in\mathbb R^2:|z|\le r\},
			$
			expands as $t\to\infty$.
		\end{remark}
		
		\begin{remark}
			As a direct corollary of Theorem~\ref{thm: SHF, local moments} and the estimate \eqref{eqSHF: E log Z two-side bound}, we obtain the following two-sided bound for the mean of the logarithm of the locally averaged SHF:
			\begin{equation}\nonumber
				-C-\frac12\log\mathbb E[\mathscr Z_t^\theta(\mathcal U_\epsilon)^2]
				\le
				\mathbb E[\log\mathscr Z_t^\theta(\mathcal U_\epsilon)]
				\le
				C-\frac12\log\mathbb E[\mathscr Z_t^\theta(\mathcal U_\epsilon)^2]
			\end{equation}
			for some $C=C(t,\theta)$. By the second-moment estimate \eqref{eqSHF: local second moments}, this is equivalent to
			\[
			\left|
			\mathbb E[\log\mathscr Z_t^\theta(\mathcal U_\epsilon)]
			+\frac12\log\log\frac{1}{\epsilon}
			\right|
			\le C.
			\]
			In \cite{GuTsai_2026_loglogCLT}, for
			$
			g_\epsilon(x)
			=
			\frac{1}{2\pi\epsilon}e^{-|x|^2/(2\epsilon)},
			$
			it is proved that
			\[
			\frac{
				\log\mathscr Z_t^\theta(g_\epsilon)
				+\frac{1+o_\epsilon(1)}{2}\log\log\frac{1}{\epsilon}
			}{
				\sqrt{\log\log\frac{1}{\epsilon}}
			}
			\]
			converges in distribution to a standard Gaussian. Our estimate for the mean of the logarithm of the locally averaged SHF suggests that the $o_\epsilon(1)$ term appearing there may be an artifact of the proof.
		\end{remark}
		
		\begin{remark}\label{remark: on GN covering number}
			In \cite{GN2026_coveringnumberSHF}, the core tool for studying the covering number problem is a size-biased concentration estimate for the
			local SHF mass given in their Corollary 8.4: there exists $c(\delta)>0$ such that for all $\epsilon >0$ small enough,
			\begin{equation}\label{eqSHF: locally tilted mass concentration}
				\mathbb E \left[
				\mathscr Z_t^\theta (\mathcal U_\epsilon ) \mathds{1}_{\mathscr Z_t^\theta (\mathcal U_\epsilon) >(\log \frac{1}{\epsilon}) ^{1/2 +\delta }} 
				\right]+ 
				\mathbb E \left[
				\mathscr Z_t^\theta (\mathcal U_\epsilon) \mathds{1}_{\mathscr Z_t^\theta (\mathcal U_\epsilon) <(\log \frac{1}{\epsilon}) ^{1/2 -\delta }} 
				\right] \le (\log \frac{1}{\epsilon})^{-c(\delta)}
			\end{equation}
			Strictly speaking, the local average considered in
			\cite{GN2026_coveringnumberSHF} is defined using a heat-kernel mollifier rather
			than the ball average $\mathcal U_\epsilon$ used here. However, this difference is not essential to the covering problem studied therein. 
			With the moment estimate \eqref{eqSHF: local moments estimate} especially for $p$ near $1$,
			we can now deduce \eqref{eqSHF: locally tilted mass concentration} in a rather direct way: for any $p>0$, 
			\[
			\mathbb E \left[
			\mathscr Z_t^\theta (\mathcal U_\epsilon) \mathds{1}_{\mathscr Z_t^\theta (\mathcal U_\epsilon) <(\log \frac{1}{\epsilon}) ^{1/2 -\delta }} 
			\right] \le (\log \frac{1}{\epsilon})^{(1/2 -\delta) p} \mathbb E [ \mathscr{ Z}_t^\theta (\mathcal U_\epsilon)^{1-p}] 
			\le C (\log \frac{1}{\epsilon})^{-\delta p +p^2 /2}.
			\]
			Take $p= \delta$, we can bound it by $(\log \frac{1}{\epsilon})^{-\delta ^2 /2}$. The same argument with the upper bound of $\mathbb E [ \mathscr{ Z}_t^\theta (\mathcal U_\epsilon)^{1+p}] $ gives 
			\[
			\mathbb E \left[
			\mathscr Z_t^\theta (\mathcal U_\epsilon ) \mathds{1}_{\mathscr Z_t^\theta (\mathcal U_\epsilon) >(\log \frac{1}{\epsilon}) ^{1/2 +\delta }} 
			\right] \le (\log \frac{1}{\epsilon})^{-\delta ^2 /2}.
			\] 
			Thus the sharp $p-$moment estimates in Theorem \ref{thm: SHF, local moments} with $p$ around $1$ imply directly that, biased by the local SHF mass, the normalized mass of an $\epsilon$-ball is concentrated at the scale
			$
			(\log \frac1\epsilon)^{1/2+o_\epsilon (1)}.
			$
		\end{remark}

		\subsection{Main ideas of the proofs}
		\label{subsec: discussion on the proofs}
		
		\subsubsection{Gaussian convexity: from localized gradient control to variance estimates}
		
		Gaussian concentration inequalities and their variants have commonly been used to study left tails and negative moments of DPRE models; see, for example, \cite{MFG_2014_1dDPREpositivity,CSZ2020_subcriticalKPZ}. In Section~3.3 of \cite{CSZ2020_subcriticalKPZ}, negative moments of partition functions for the two-dimensional DPRE in the subcritical regime are obtained from a lower-tail estimate for the free energy derived via a localized concentration inequality. The argument proceeds by showing that the gradient of the free energy is uniformly controlled on an event whose probability is bounded away from zero. The crucial input is the strict subcriticality of the disorder strength $\beta_N$ in \eqref{eqDPRE: beta_N, subcritical}. The gap from the critical threshold $\pi$ allows one to increase the effective coupling slightly while remaining in the $L^2$ regime, thereby obtaining the replica-overlap estimates needed for the localized gradient bound. In the critical window \eqref{eqDPRE: beta_N, critical}, however,
		$
		\beta_N^2\log N\longrightarrow\pi,
		$
		so this margin disappears, and the same overlap-based argument no longer provides uniform control of the gradient.
		
		The Gaussian convexity property developed in \cite{AC_2015_convex_positive_lambda}, offers a different mechanism. For a convex functional $F(G)$ of a Gaussian vector $G$, classical Gaussian concentration produces bounds on the scale of the squared Lipschitz constant of $F$. By contrast, Theorem~1 and its corollary in \cite{Chen2026_Ehrhard_convex1} replace this Lipschitz scale by the variance of $F(G)$. This improvement is particularly useful when $F(G)$ exhibits superconcentration (refer to  \cite{Chatterjee2014_superconcentration}), that is, when $\var ( F(G))$ is much smaller than the square of the Lipschitz constant of $F$. 
		
		There is also a difference in the logical order of the two approaches. In \cite{CSZ2020_subcriticalKPZ}, a lower-tail estimate is proved first and then integrated to obtain negative-moment bounds. In \cite{Chen2026_Ehrhard_convex1}, convexity of the normalized log-moment generating function first yields estimates for negative exponential moments, from which the lower-tail bound follows as a corollary. Motivated by the latter approach, we exploit the fact that the DPRE free energy is a convex function of the Gaussian disorder and subsequently transfer the resulting moment and lower-tail estimates to the critical SHF. We also mention that the study of left tail for convex-Gaussian random variables has appeared in for example \cite{GV_2018_convexGaussian_lefttail} and \cite{Valettas2019_convexGaussian_lefttail}. More detailed discussion can be seen in Subsection \ref{subsection: Ehrhard ineq, R.Vs}.
		
		\subsubsection{Gaussian convexity property and Ehrhard-type moments comparison inequality}
		We first introduce the Gaussian convexity property developed in \cite{AC_2015_convex_positive_lambda}, \cite{Guerra_2022_convex_positive_lambda} and \cite{Chen2026_Ehrhard_convex1}, and then specialize it to DPRE and SHF. As we will see, the moment comparison inequality provides a powerful tool for converting a limited amount of moment information into estimates for other moments.
		
		For a random variable $X$ with finite exponential moments, define its normalized log-moment generating function by
		\begin{equation}\label{eq: normalized log MGF}
			\Lambda_X(\lambda)
			:=
			\begin{cases}
				\displaystyle
				\frac{1}{\lambda}
				\log\mathbb E[e^{\lambda X}],
				& \lambda\ne0,\\[1.2ex]
				\mathbb E[X],
				& \lambda=0.
			\end{cases}
		\end{equation}
		
		Let $G\sim\mathcal N(0,I_d)$ and let $F:\mathbb R^d\to\mathbb R$ be convex. It is proved in \cite{Chen2026_Ehrhard_convex1} that, provided
		\[
		\mathbb E[e^{\lambda F(G)}]<\infty
		\qquad\text{for every }\lambda>0,
		\]
		the map $\lambda\mapsto\Lambda_{F(G)}(\lambda)$ is convex on $\mathbb R$, which extends the same convexity result but only for positive $\lambda$ and Gaussian spin-glass setting in \cite{AC_2015_convex_positive_lambda} and \cite{Guerra_2022_convex_positive_lambda}.
		
		For the convenience of discussion here, we assume
		$
		\mathbb E[e^{F(G)}]=1,
		$
		so that $\Lambda_{F(G)}(1)=0$. Convexity of $\Lambda_{F(G)}$ implies that
		$
		\lambda\mapsto 
		\frac{\Lambda_{F(G)}(\lambda)}{\lambda-1}
		$
		is nondecreasing. Equivalently,
		\begin{equation}\label{eqG: Gaussian moments comparison}
			\frac{\log\mathbb E[e^{pF(G)}]}{p(p-1)}
			\le
			\frac{\log\mathbb E[e^{qF(G)}]}{q(q-1)},
			\qquad
			\forall\,p<q,\quad p,q\notin\{0,1\},
		\end{equation}
		and
		\begin{equation}\label{eqG: Gaussian moments comparison 2}
			-\frac{\log\mathbb E[e^{qF(G)}]}{q(q-1)}
			\le
			\mathbb E[F(G)]
			\le
			-\frac{\log\mathbb E[e^{pF(G)}]}{p(p-1)},
			\qquad
			\forall\,p<0,\quad q>0.
		\end{equation}
		
		We next apply this comparison to DPRE. Let $f\in C_c(\mathbb R^2)$ be non-negative with $\|f\|_{L^1(\mathbb R^2)}>0$. Recall that $\mathcal Z_{\llbracket2Nt\rrbracket}^{\beta_N(\theta)}(f)$ is the two-dimensional DPRE partition function in the critical window tested against $f$, and that it converges in distribution to $\mathscr Z_t^\theta(f)$. In Lemma~\ref{lemma: convexity of Free energy}, we show that the free energy
		$
		F_N(\omega)
		:=
		\log\mathcal Z_{\llbracket2Nt\rrbracket}^{\beta_N(\theta)}(f)
		$
		is a convex function of a finite-dimensional Gaussian vector and therefore falls within the framework above. Applying \eqref{eqG: Gaussian moments comparison} with $p<0$ bounds the $p$-th negative moment of
		$
		\frac{\mathcal Z_{\llbracket2Nt\rrbracket}^{\beta_N(\theta)}(f)}
		{\mathbb E[\mathcal Z_{\llbracket2Nt\rrbracket}^{\beta_N(\theta)}(f)]}
		$
		by the appropriate power $p(p-1)/2$ of its second moment. Using the moment convergence result \eqref{eqSHF: moments limit from Z_N, h =1,2} together with a truncation argument, and then letting $N\to\infty$, yields the upper bound for every negative moment of $\mathscr Z_t^\theta(f)$. The same argument applies for $p\in(1,2)$. This proves \eqref{eqSHF: p<0, 1<p<2 moments upper bounds} in Theorem~\ref{thm: SHF moment comparison ineq}.
		
		To make it more convenient for using the moment comparison property above for weak limits of convex-Gaussian random variables like SHF, we turn to the Ehrhard-type argument. The recent work \cite{F2025_Gaussian_convex_inequality} proves the same convexity of normalized log-moment generating functions for Ehrhard-type random variables, which are ''equivalent'' to convex functionals of Gaussian vectors; see Lemma~\ref{lemma: Ehrhard and convex Gaussian} for details. Inspired by \cite{F2025_Gaussian_convex_inequality}, we extend the moment comparison inequality to weak limits of convex functionals of Gaussian vectors by proving that the class of Ehrhard-type random variables is closed under weak convergence; see Lemma~\ref{lemma: Ehrhard type is close} and Theorem~\ref{thm: Ehrhard-type convexity}. We refer to the resulting inequality as the \emph{Ehrhard-type moment comparison inequality}. In particular, from the weak convergence of $\log\mathcal Z_{\llbracket2Nt\rrbracket}^{\beta_N(\theta)}(f)$ to $\log\mathscr Z_t^\theta(f)$, we obtain
		\begin{equation}\label{eqIntro: SHF moments comparison}
			\frac{\log\mathbb E[\mathscr Z_t^\theta(f)^p]}{p(p-1)}
			\le
			\frac{\log\mathbb E[\mathscr Z_t^\theta(f)^q]}{q(q-1)},
			\qquad
			\forall\,p<q,\quad p,q\notin\{0,1\}
		\end{equation}
        for non-negative $f$ with $||f||_1=1$.
		
		We now illustrate how the moment comparison inequality yields estimates for all real moments in terms of the second moment. First, \eqref{eqIntro: SHF moments comparison} directly gives estimates in one direction; see Theorem~\ref{thm: intro, Ehrhard moment comparison ineq}. The remaining directions are obtained by choosing suitable comparison exponents in \eqref{eqIntro: SHF moments comparison}. For a non-integer $p>2$, comparison with an integer $h>p$ reduces an upper bound for $\mathbb E[\mathscr Z_t^\theta(f)^p]$ to an upper bound for the integer moment $\mathbb E[\mathscr Z_t^\theta(f)^h]$. Similarly, suitable comparisons reduce the upper bound for $p\in(0,1)$ and the lower bound for $p\in(1,2)$ to lower bounds for negative moments. Consequently, upper bounds for positive integer moments, together with suitable lower bounds for negative moments, propagate to sharp two-sided estimates for every real moment.
		
		The moment comparison framework developed here can be adapted readily to other Gaussian models and their scaling limits, including Gaussian multiplicative chaos and mollified stochastic heat equations. Further details are given in Section~\ref{section: GMC and mSHE}.
		
		\subsubsection{Moments of locally averaged SHF}
		The moment comparison principle alone cannot give all the real moments estimate on both sides. To complete the picture, we still need the upper bounds for integer moments of order $h\ge2$ and lower bounds for negative local moments. Both estimates require additional information specific to SHF.
		
		We first discuss the positive integer moments. 
		For integer $h$, \cite{LZ2024_micromomentSHF} exploits the exact moment formula for SHF, expressed as a series of integrals indexed by collision diagrams, and proves the two-sided bounds in \eqref{eqSHF: LZ local moments}. The upper-bound proof in \cite{LZ2024_micromomentSHF} estimates the collision-diagram integrals term by term and then applies a combinatorial argument. This produces an additional factor
		$
		O\left(
		|\log\epsilon|^{1/\log^{(3)}(\epsilon^{-1})}
		\right).
		$
		
		Instead of estimating the collision series term by term, in Subsection~\ref{subsection: int local moments, upper bound} we control the whole series through a bootstrap argument. Specifically, we rewrite the series \eqref{eq: M_{h,epsilon}^theta <}, originally obtained in \cite{LZ2024_micromomentSHF}, in the equivalent form \eqref{eqU: represtt of upper bounds for M_epsilon}, whose main part is
		$
		\sum_{k=0}^\infty(qK_\lambda)^k\mathbf 1(x),
		$
		where $q>0$, $\lambda>0$, and $K_\lambda$ is the positive linear operator defined in \eqref{eq: K_lambda f}. Rather than estimating the summands $((qK_\lambda)^k\mathbf 1)(x)$ individually, define
		$
		S_N(x)
		:=
		\sum_{k=0}^N((qK_\lambda)^k\mathbf 1)(x).
		$
		Then
		$
		\mathbf 1+qK_\lambda S_N=S_{N+1}.
		$
		If we can construct some function $H$ satisfying
		$
		\mathbf 1+qK_\lambda H\le H
		$
		and the initial bound $S_0=\mathbf 1\le H$, then induction gives $S_N\le H$ for all $N$, and hence the same bound in the limit $N\to\infty$. At the same time, we choose $H$ carefully so that it satisfies the integral bound \eqref{eq: int H(x)/x}, which yields the sharp moment estimate. This strategy was partially inspired by the bootstrap argument used in the study of the left tail of Gaussian multiplicative chaos in \cite[Lemma~4.5]{DS2010_liouville_and_KPZ}, where a bootstrap inequality for the left tail is first obtained and then iterated against a suitable supersolution to derive the desired decay estimate.
		
		We next turn to lower bounds for the negative local moments of SHF. Following the block decomposition of \cite{GuTsai_2026_loglogCLT}, we divide the time interval $[\epsilon^2,1]$ geometrically by setting
		\[
		t_i:=\epsilon^2b^{-2i},
		\qquad
		N:=\left\lfloor
		\frac{\log\epsilon^{-1}}{\log b^{-1}}
		\right\rfloor,
		\]
		where $b\in(0,1)$ is fixed and sufficiently small. By the flow property of SHF (refer to \cite{ClarkMian2026_SHFflow}), 
		$
		\left\{
		\mathscr Z_{\epsilon^2,t}^{\theta}
		(g_{\epsilon^2},\mathbf 1)
		\right\}_{t\ge\epsilon^2}
		$
		is a non-negative martingale. Since $x\mapsto x^p$ is convex for $p<0$, this martingale property reduces the desired lower bound at time $1$ to the corresponding estimate at time $t_n$ for some $n<N$. The reason for discarding the final finitely many blocks is the same as in \cite{GuTsai_2026_loglogCLT}: we have uniform moment control only for the early-time segments; see Lemma~\ref{prop:E U_i^4 <}, Lemma~\ref{lemma: Block variance}, and Remark~\ref{rmk:early segment}. The decoupling expansion of \cite[Proposition~2.1]{GuTsai_2026_loglogCLT} gives
		\[
		\mathscr Z_{\epsilon^2,t_n}^{\theta}
		(g_{\epsilon^2},\mathbf 1)
		=
		\prod_{i=1}^{n}
		Z_i(g_{t_{i-1}},\mathbf 1)
		+R_n,
		\]
		where the leading term is a product of independent block contributions and the remainder $R_n$ collects the interactions between different blocks. We restrict to an event on which all block fluctuations are moderate. On this event, the negative moment of the product can be estimated sharply, while the ratio of the remainder to the product is shown to be small.
		
		To conclude this discussion, we compare our argument in Subsection~\ref{subsection: Local negative moments lower bound} with the recent preprint \cite{Huang2026_fractionalSHF}. Both arguments build on the geometric time  decomposition of \cite{GuTsai_2026_loglogCLT} and approximate the SHF mass by a product of independent block masses. There are, however, two important differences. First, our argument concerns lower bounds for negative moments, whereas \cite{Huang2026_fractionalSHF} studies $p$-th moment for $p\in (0,1)$. Second, the block-variance estimate in \cite[(2.6)]{Huang2026_fractionalSHF}, quoted from \cite[(2.15)]{GuTsai_2026_loglogCLT}, gives only 
		\[ 
		\mathbb E[U_i^2]
		=
		\frac{1+O(\rho(b))}{N-i},
		\qquad
		\rho(b)\longrightarrow0
		\quad\text{as }b\downarrow0.
		\]
		This estimate is sufficient to identify the logarithmic moment exponent $p(p-1)/2$ only up to an $o(1)$ error. In Lemma~\ref{lemma: Block variance}, we use a finer estimate that recovers the exact exponent in the block-product estimate without introducing an additional $o(1)$ loss in the exponent.
		
		\subsection{Structure of the paper}
		
		Section~\ref{section: moment comparison} introduces the Ehrhard-type moment comparison inequality that provides the general framework for the paper. Section~\ref{section: preliminary, SHF} recalls basic facts about DPRE and SHF. In Section~\ref{section: negative moments proof}, we prove Theorems~\ref{thm: SHF moment comparison ineq} and \ref{thm: left tail of SHF}, concerning the moment comparison inequality and the left tail of SHF. In Section~\ref{section: local moments of SHF}, we prove Theorem~\ref{thm: SHF, local moments} on the asymptotics of the local moments of SHF. Finally, Section~\ref{section: GMC and mSHE} discusses applications of the method to GMC and the mollified stochastic heat equation.
		
		\paragraph{\textbf{Notation.} }
		We write the heat kernel as
		\[
		g_t(x)
		:=
		\frac{1}{2\pi t}
		e ^{-\frac{|x|^2 }{2t} }
		\quad
		t>0,\quad x\in\mathbb R^2,
		\]
		where $|x|$ denotes the Euclidean norm of $x\in\mathbb R^2$. We denote the standard Gaussian distribution by $\mathcal N(0,1)$.
		
		\section{Ehrhard inequality and moment comparison inequality}\label{section: moment comparison}
		
		\subsection{Ehrhard inequality and Ehrhard-type random variables}\label{subsection: Ehrhard ineq, R.Vs}
		
		The models considered in this paper are all the convex functionals of Gaussian vectors or their limits. To employ this convex-Gaussian structure, let us introduce the Ehrhard inequality first:  
		\begin{theorem}[Ehrhard inequality \cite{Ehrhard1983}]
			Set $\Phi(a):= \mathbb P ( \mathcal N(0,1) <a)$ for $a\in \mathbb R \cup \{\infty, -\infty\}$. For any $n\in \mathbb Z_{>0}$, write $\gamma _n$ as the standard Gaussian measure on $\mathbb R^n$. Then, for any two convex sets $A$, $B\subset \mathbb R^n$ and $\lambda \in (0,1)$,
			\[
			\Phi^{-1} \circ \gamma _n (\lambda A + (1-\lambda) B) \ge \lambda \Phi^{-1} \circ \gamma_n(A) + (1-\lambda)  \Phi^{-1} \circ \gamma_n(B).
			\]
		\end{theorem}
		In \cite{Borell2003_Ehrhard}, the inequality above is extended for any two Borel sets. As a direct corollary, we have 
		\begin{corollary}[Lemma 3.1 in \cite{F2025_Gaussian_convex_inequality}]\label{corollary: Ehrhard for convex functional}
			Assume $F: \mathbb R^n \to \mathbb R \cup \{ \infty\}$ is a convex function and $G$ is an $n$-dimensional standard Gaussian vector. Then
			\begin{equation}\label{eqEharhard: concave}
				r\mapsto \Phi^{-1} \circ \mathbb P(F(G) <r) \text{ is concave.}
			\end{equation}
		\end{corollary} 
		
		Heuristically, since $\Phi ^{-1}$ is an increasing function, the concavity of $\Phi^{-1} \circ \mathbb P(F(G) <r)$ implies the left tail $P(F(G) <r)$ should decrease rather fast when $r\to -\infty$. Indeed, in \cite{GV_2018_convexGaussian_lefttail} and \cite{Valettas2019_convexGaussian_lefttail}, Ehrhard inequality is used to deduce the following left tail for convex functionals of Gaussian vectors. 
		
		\begin{theorem}[ \cite{Valettas2019_convexGaussian_lefttail}]
			Assume $F: \mathbb R^n \to \mathbb R \cup \{ \infty\}$ is a convex function, $G$ is an $n$-dimensional standard Gaussian vector and $\mathbb E [ |F(G) |] <\infty$. Write $M$ as the median of $F(G)$. Then 
			\[
			\mathbb P \left(
			F(G) -M < -t \mathbb E [(F(G) -M)_+]
			\right) \le \mathbb P \left(\mathcal N(0,1) >\frac{t}{\sqrt{2\pi}} \right),\quad t\ge 0.
			\]
		\end{theorem}
		
		As pointed out in Section 1.1 of \cite{Chen2026_Ehrhard_convex1}, by using 
		\[
		\mathbb E [(F(G) -M)_+] \le \sqrt{\var F(G)},\quad M \ge \mathbb E F(G) -\sqrt{\var F(G)},  
		\]
		we obtain a left tail of $F(G)$ sensitive to its variance if we suppose also $\var [F(G)]<\infty$: 
		\begin{equation}\label{eqEhrhard: F(G) left tail 1}
			\mathbb P \left(
			F(G) -\mathbb E [F(G)] < -t \sqrt{\var F(G) }
			\right)\le \mathbb P ( \mathcal N(0,1) > \frac{t-1}{\sqrt{2\pi} } ),\quad t\ge 1.
		\end{equation}
		With $x e^{x^2 /2} \mathbb P ( \mathcal N(0,1) > x)\to \frac{1}{\sqrt{2\pi} }$, the left tail in \eqref{eqEhrhard: F(G) left tail 1} decreases as $e^{-t^2 /(4\pi) + O(t) }$ as $t\to +\infty$. 
		Especially, the left tail \eqref{eqEhrhard: F(G) left tail 1} obtained from the Ehrhard inequality is good enough to deduce the finiteness of each negative moment of $e^{F(G) }$. In \cite{Chen2026_Ehrhard_convex1}, the decreasing rate of left tail is further improved to be $e^{-t^2 /2}$ as a corollary of Gaussian convexity inequality we will introduce in the next subsection. 
		
		In what follows, let us discuss whether it is possible to extend the concavity in \eqref{eqEharhard: concave} for more general random variables instead of only convex functionals of Gaussian vectors. To make the discussion here clear, we give a rigorous definition first:
		\begin{definition}[Ehrhard-type random variables]\label{def: Ehrhard-type random variables}
			Let $\Phi(a):= \mathbb P (\mathcal N(0,1)<a)$ for any $a\in \mathbb R$. Define its generalized inverse function on $[0,1 )$:
			\[
			\widehat{\Phi}^{-1}(u):=
			\begin{cases}
				\Phi^{-1}(u),&u\in(0,1),\\
				-\infty,&u=0.
			\end{cases}
			\]
			A real-valued random variable $X$ is said to be Ehrhard-type if 
			\begin{enumerate}
				\item $X$ is not bounded from above:
				\[
				\mathbb P(X<t)<1,
				\qquad t\in\mathbb R,
				\] 
				\item 
				$
				h_X(t):=
				\widehat{\Phi}^{-1}\bigl(\mathbb P(X<t)\bigr)
				$
				is a concave function from $\mathbb R$ to
				$\mathbb R\cup\{-\infty\}$.
			\end{enumerate}
		\end{definition}
		By Corollary \ref{corollary: Ehrhard for convex functional}, we know that a convex functional of some Gaussian vector is Ehrhard-type. Reversely, the following lemma shows that  an Ehrhard-type random variable must have the same law as a convex functional of a standard Gaussian random variable.
		
		\begin{lemma}[Theorem 4.3 in \cite{F2025_Gaussian_convex_inequality}]\label{lemma: Ehrhard and convex Gaussian}
			Assume $X$ is an Ehrhard-type random variable. Then there is a $\mathbb R$-valued convex function $f$ such that $X$ has the same law as $f(g)$ {, where} $g$ {is} a standard Gaussian random variable. In particular, 
			\begin{align*}
				\bbE[X]>-\infty,
				\quad\text{ and }\quad 
				\bbE[e^{pX}]<\infty
				\quad 
				\text{ for all }p<0.
			\end{align*}
		\end{lemma}
		
		For completeness, we give a self-contained proof here based on the idea of generalized reverse function \eqref{eqEhrhard: generalized inverse, f} which was once used in \cite{Valettas2019_convexGaussian_lefttail} under the name 'Gaussian rearrangement'. 
		\begin{proof}
			Recall $h_X$ defined in Definition \ref{def: Ehrhard-type random variables}. Set its generalized inverse function as
			\begin{equation}\label{eqEhrhard: generalized inverse, f}
				f(z):= \inf \left\{
				t\in \mathbb R \cup \{ -\infty , \infty\} : h_X(t) >z
				\right\}
			\end{equation}
			which is non-decreasing and convex since $h_X$ is non-decreasing and concave. We check that $f$ is $\mathbb R$-valued by contradiction. Assume {that } there {exists} $z\in \mathbb R$ such that $f(z) =\infty$. Then for any $t\in \mathbb R$, $h_X(t)\le z$, and thus $\mathbb P (X<t)\le \Phi(z) <1$ for any $t$, which contradicts no-atomness at infinity {$X$ is real-valued} in definition \ref{def: Ehrhard-type random variables}. Similar discussion gives $f(z)$ is not $-\infty$ for any $z\in \mathbb R$.  
			
			We claim that for any $z,\ t\in \mathbb R$, $f(z) <t$ is equivalent to $h_X (t)>z$. Indeed, if $h_X(t) >z$, we have $P(X<t)>0$. By the left continuity of $t\mapsto P(X<t)$, we have $h_X(t-\Delta t) >z$ for some small $\Delta t >0$. Thus $f(z) \le t-\Delta t<t$. The reverse direction simply {follows} from the definition of $f(\cdot)$.   
			
			As a result of the equivalence above, for a standard Gaussian random variable $g$, we have 
			\[
			\mathbb P (f(g) <t) = P (g<h(t) ) =\Phi \circ \Phi^{-1} \circ \mathbb P (X<t) = \mathbb P (X<t),\quad \forall t\in \mathbb R.
			\]
			Thus $X$ has the same law as $f(g)$. Since $f$ is convex and finite, there are $a,\ b\in \mathbb R$ such that $f(z)\ge az +b$ for all $z\in \mathbb R$. For any $p<0$, we have 
			\[
			\mathbb E [e^{pX} ] =\mathbb E [e^{pf(g)}] \le \mathbb E 
			[ e^{p(ag+b)}]<\infty .
			\]
			{Due to the same reason}, we get 
			\[
			\mathbb E [X ] = \mathbb E [f(g) ] \ge \mathbb E [(ag + b)] >-\infty.
			\]
		\end{proof}
		
		By Lemma \ref{lemma: Ehrhard and convex Gaussian}, Ehrhard-type property cannot enlarge the class of distribution in Theorem \ref{thm: Ehrhard-type convexity} beyond convex functionals of Gaussians. However, we show in the following Lemma that Ehrhard-type property is closed under the weak limit, which makes it well suited for convex-Gaussian limits like SHF and GMC discussed in this paper.
		
		\begin{lemma}\label{lemma: Ehrhard type is close}
			{Let} $\{X_N: N =1,...\}$ be Ehrhard-type. If $X_N$ converges weakly to a non-constant random variable $X$, then $X$ is also Ehrhard-type. 
		\end{lemma}
		\begin{proof}
			We firstly prove that $P(X<x)\ne 1$ for all $x\in \mathbb R$ by contradiction. Assume $P(X<x) = 1$ for some $x\in \mathbb R$. Since $X$ is not constant by our assumption, we can {choose} two continuity points {from its distribution function}, say $x_1<x_2$, satisfying $P(X<x_i) \in (0,1)$ for $i\in \{1,2\}$. Since $X_N$ are Ehrhard-type, we have 
			\[
			h_{X_N}(x) \le  
			h_{X_N}(x_2) + \frac{x-x_2}{x_2-x_1} \left(
			h_{X_N}(x_2)-h_{X_N}(x_1)
			\right),\quad { \forall x\ge x_2}
			\]
			with $h_{X_N} (\cdot) := \hat \Phi ^{-1} \circ \mathbb P(X_N <\cdot)$. Then by weak convergence,
			\[
			\begin{aligned}
				& \limsup_N h_{X_N}(x) 
				\le h_{X}(x_2) + \frac{x-x_2}{x_2-x_1} \left(
				h_{X}(x_2)-h_{X}(x_1)
				\right)<\infty. 
			\end{aligned}
			\]
			However, by Portmanteau lemma, 
			\[
			\liminf P (X_N <x) \ge P(X<x) =1,
			\]
			so $h_{X_N}(x)\rightarrow\infty$, which is a contradiction.  
			
			Next, we prove the concavity of $h_X(\cdot)$. Fix any $x_0$, $x_1$, and $x_\lambda := \lambda x_0 + (1-\lambda) x_1$ for some $\lambda \in (0,1)$. If $P(X<x_0)=0$, we have 
			\[
			h_X(x_\lambda) \ge \lambda h_X (x_0) + (1-\lambda ) h_X (x_1)
			\]
			with $h_X$ in Definition \eqref{def: Ehrhard-type random variables}, since the right hand side is $-\infty$. The same {discussion} holds for {the cases} $P(X<x_\lambda)=0$ or $P(X<x_1) =0$. As a result, it suffices to consider the case when all these three probabilities belong to $(0,1)$. Since the set of atoms for $X$ is countable, we can take a sequence of continuity points for the law of $X$, say 
			\begin{equation}\label{Ehrhard: continuity points of law of X}
				\{ x_0-\epsilon_k,\ x_1-\epsilon_k,\ x_\lambda -\epsilon_k:\ k=1,...\}
			\end{equation}
			with $\epsilon_k \downarrow 0$. Since $\{X_N: N=1,...\}$ are Ehrhard-type random variables, we have 
			\[
			h_{X_N}(x_\lambda -\epsilon_k) \ge 
			\lambda h_{X_N}(x_0 -\epsilon_k) +(1-\lambda)
			h_{X_N}(x_1 -\epsilon_k). 
			\] 
			We let $N\to \infty$. Since $X_N$ weakly converges to $X$ and all the points in \eqref{Ehrhard: continuity points of law of X} are {continuity points of the distribution function of} $X$, we have 
			\[
			h_X(x_\lambda -\epsilon_k) \ge 
			\lambda h_X(x_0 -\epsilon_k) +(1-\lambda)
			h_X(x_1 -\epsilon_k). 
			\]
			By the continuity of $\Phi^{-1}$ on $(0,1)$ and also the left continuity of $x\mapsto P(X <x)$, letting $\epsilon_k \downarrow 0$, we get 
			\[
			h_X(x_\lambda ) \ge 
			\lambda h_X(x_0 ) +(1-\lambda)
			h_X(x_1 )
			\]
			which finishes the proof. 
		\end{proof}

		\subsection{Gaussian convexity property and Ehrhard-type moments comparison inequality}\label{subsection: Ehrhard-type moment comparison}
		The main goal of this subsection is to introduce the Gaussian convexity property  and its corollary Ehrhard-type moments comparison inequality which serves as a powerful tool to estimate the moments of Ehrhard-type random variables (see Definition \ref{def: Ehrhard-type random variables}). To start with, we introduce the Gaussian convexity property from \cite{Chen2026_Ehrhard_convex1}.
		\begin{theorem}[Theorem 1 in \cite{Chen2026_Ehrhard_convex1}] \label{thm: Chen, thm 1 on convexity}
			Assume $F$ is a convex function on $\mathbb R^d$ and $G$ is a standard finite-dimensional Gaussian vector. Assume $\mathbb E [e^{\lambda F(G)}]<\infty $ for every $\lambda >0$. Then the function 
			\begin{equation}\label{eqChen: Phi function}
				\Lambda_{F(G)} (\lambda):= \begin{cases}
					\frac{\log \mathbb E [e^{\lambda F(G)}]}{\lambda} & \text{ if } \lambda \ne 0\\
					\mathbb E [F(G)] & \text{ if } \lambda = 0
				\end{cases}
			\end{equation}
			is $\mathbb R$-valued and convex. 
		\end{theorem}
		The convexity above was firstly obtained for all positive $\lambda$ in for example \cite{AC_2015_convex_positive_lambda} and \cite{Guerra_2022_convex_positive_lambda} and then extended to the whole $\lambda \in \mathbb R$ in \cite{Chen2026_Ehrhard_convex1}. As a corollary of Theorem \ref{thm: Chen, thm 1 on convexity}, Chen obtained an improved left tail compared with \eqref{eqEhrhard: F(G) left tail 1} for $F(G)$: Formally, $\frac{\td}{\td \lambda }\Lambda_{F(G)} (0) = \frac{1}{2}\var [F(G) ]$, and the left tail in \eqref{eqChen: Gaussian left tail} follows from the Markov inequality and the estimate for $\mathbb E [e^{\lambda F(G)}]$ with $\lambda <0$ obtained from the convexity of \eqref{eqChen: Phi function}. We refer to \cite{Chen2026_Ehrhard_convex1} for the rigorous proof.
		\begin{theorem}[Theorem 3 in \cite{Chen2026_Ehrhard_convex1}]\label{thm: Chen, theorem 3, left tail}
			Assume $F$ is a convex function on $\mathbb R^d$ and $G$ is a standard finite-dimensional Gaussian vector. Assume $\mathbb E [F(G)^2]<\infty$, then 
			\begin{equation}\label{eqChen: Gaussian left tail}
				\mathbb P \left(
				F(G) - \mathbb E [F(G)] \le -t  
				\right)\le e^{-\frac{t^2}{2 \var ( F(G) )} },\quad \forall t\ge 0.
			\end{equation}
		\end{theorem}
		
		Recall that Lemma \ref{lemma: Ehrhard and convex Gaussian} gives the 'equivalence' of convex-Gaussian and Ehrhard-type random variables. We immediately get the following Ehrhard-type convexity inequality from Theorem \ref{thm: Chen, thm 1 on convexity}. We note that a different proof is given in Theorem 1 of \cite{F2025_Gaussian_convex_inequality} without using the result in \cite{Chen2026_Ehrhard_convex1}. 
		\begin{theorem} \label{thm: Ehrhard-type convexity}
			Assume a random variable $X$ is Ehrhard-type. Then
			\begin{equation}\label{eqEhrhard: Lambda_X}
				\Lambda _X (p):= \begin{cases}
					\frac{\log \mathbb E [e^{p X}]}{p} & \text{ if } p \ne 0\\
					\mathbb E [ X] & \text{ if } p = 0
				\end{cases}
			\end{equation}
			is a $\mathbb R \cup \{ \infty\}$-valued convex function.
		\end{theorem} 
		
		\begin{proof}
			We first prove $\Lambda _X(p) >-\infty$ for any $p\in \mathbb R$.
			From definition \ref{def: Ehrhard-type random variables}, $X$ is not identically $-\infty$, thus $\Lambda _X (p)> -\infty$ for all $p>0$. By Lemma \ref{lemma: Ehrhard and convex Gaussian}, $\Lambda _X (p)> -\infty$ for all $p\le 0$. \par 
			Now we prove that $\Lambda_X(\cdot)$ is convex. Since $X$ is Ehrhard-type, by Lemma \ref{lemma: Ehrhard and convex Gaussian}, $X$ has the same law as $f(g)$ for $g$ a standard Gaussian random variable and $f$ a $\mathbb R$-valued convex function. Take affine functions $\ell_k\le f$ such that $f=\sup_k \ell_k$, and define $f_n:= \sup_{1\le k\le n} \ell_k$, which is globally Lipschitz and convex for each $n$. Especially, $\mathbb E e^{pf_n(g)}<\infty$ for every $p>0$. By Theorem \ref{thm: Chen, thm 1 on convexity},  
			\[
			\Lambda _n (p):= \begin{cases}
				\frac{\log \mathbb E [e^{p f_n(g) }]}{p} & \text{ if } p \ne 0\\
				\mathbb E [ f_n(g) ] & \text{ if } p = 0
			\end{cases}
			\]
			is $\mathbb R$-valued and convex. For $p> 0$, by the monotone convergence theorem, $\mathbb E [e^{pf_n(g)}] \uparrow \mathbb E [e^{pf(g) }]= \mathbb E [e^{pX}]$. For $p<0$, we have $e^{pf_1(g)} \ge e^{p f_n(g) }$ for all $n$. Then the dominated convergence theorem gives $\mathbb E [e^{pf_n(g)}] \downarrow \mathbb E [e^{ pf(g)}] = \mathbb E [e^{pX}]$. For $p=0$,
			again using monotone convergence and $f_n\uparrow f$, we get $\Lambda _n(0) \to \Lambda (0)$. Therefore, $\Lambda _X(p)$ is a pointwise limit of convex functions $\Lambda_n(p)$, so it is also convex. 
		\end{proof}
		
		As a direct corollary of Theorem \ref{thm: Ehrhard-type convexity}, we get the \emph{Ehrhard-type moment comparison inequality} in Theorem \ref{thm: intro, Ehrhard moment comparison ineq} which, together with Lemma \ref{lemma: Ehrhard type is close}, provides a general framework for the moments estimate of different models in this paper. For convenience of discussion, we rewrite Theorem \ref{thm: intro, Ehrhard moment comparison ineq} here: 
		\begin{theorem}\label{thm: Ehrhard moment comparison ineq}
			Assume a real-valued random variable $X$ is Ehrhard-type with $M:=e^X$ satisfying $\mathbb E [M]=1$. Then 
			\begin{equation}\label{eqMC: moment comparison ineq}
				\frac{\log \mathbb E [M^p]}{(p-1) p} \le \frac{\log \mathbb E [M^q]}{(q-1) q},\quad \forall p<q,\ p,q\notin \{ 0,1\} 
			\end{equation} 
			and
			\begin{equation}\label{eqMC: E X two-side}
				- \frac{\log \mathbb E [ M^q]}{q(q-1)}
				\le 
				\mathbb E [X] 
				\le - \frac{\log \mathbb E [ M^p]}{p(p-1)},\quad \forall p<0,\ q>0.
			\end{equation} 
			In particular, 
			\begin{enumerate}
				\item When $p\in (-\infty ,0) \cup (1,2)$ and $q=2$,
				$
				\mathbb E [M^p] \le (\mathbb E [ M^2] )^{p(p-1)/2}.
				$
				\item When $p\in (0,1)$, $q=2$, 
				$
				\mathbb E [M^p] \ge (\mathbb E [ M^2] )^{p(p-1)/2}.
				$
				\item When $p=2$ and $q>2$
				$
				\mathbb E [M^q] \ge  (\mathbb E [ M^2] )^{q(q-1)/2}.
				$	
			\end{enumerate}
		\end{theorem}
		\begin{proof}
			By the convexity of $\Lambda_X(\cdot)$ defined in \eqref{eqEhrhard: Lambda_X} and $\Lambda_X (1)= \log \mathbb E [e^X]=0$, we get $r\mapsto \Lambda_X(r)/(r-1)$ is non-decreasing and then \eqref{eqMC: moment comparison ineq} and \eqref{eqMC: E X two-side} follows directly. 
		\end{proof}
		
		\begin{remark}
			The condition $\mathbb E [e^X]=1$ in Theorem \ref{thm: Ehrhard moment comparison ineq} is only for convenience of statement. For the general case with $\mathbb E [e^X]\in (0,\infty)$, we can apply the same argument for the normalized $\bar X := {X} -{\log \mathbb E [e^X]}$. 
		\end{remark}
		
		\begin{remark}
			It is pointed out in \cite{F2025_Gaussian_convex_inequality} that, the convexity of normalized log moment generating function \eqref{eqEhrhard: Lambda_X} is not restricted to convex functionals of Gaussian vectors. A typical example is shown in \cite{F2025_Gaussian_convex_inequality}: if $X$ is a Poisson random variable with exponent $a$, then $\frac{\log \mathbb E [e^{\lambda X}]}{\lambda}= \frac{a(e^\lambda -1)}{\lambda}$ is convex although $X$ is not convex-Gaussian. It would be interesting to extend the framework here to non-Gaussian cases with similar moment comparison inequalities. 
		\end{remark}
		
		\section{Preliminaries for DPRE and Critical 2D SHF}\label{section: preliminary, SHF}
		
		Recall that we have defined the point-to-point partition function of $d-$dimensional DPRE model in \eqref{eqDPRE: Z_{m,n}^{beta,omega}}: Write $\{\omega(t,x):t\in \mathbb Z_+,\ x\in \mathbb Z^d \}$ the i.i.d. standard Gaussian random variables. For any $m,\ n\in \mathbb Z_+$, $\beta >0$, 
		\begin{equation}\label{eqDPRE: Z_{m,n}^{beta,omega}, preliminary}
			\mathcal Z_{m,n}^{\beta,\omega}(x,y)
			:=
			\mathrm E\!\left[
			\exp\!\left\{
			\sum_{k=m+1}^{n-1}
			\left(
			\beta\omega(k,S_k)-\frac{\beta^2}{2}
			\right)
			\right\}
			\mathbf 1_{\{S_n=y\}}
			\,\middle|\, S_m=x
			\right]
		\end{equation}
		{where the above expectation is taken over the simple symmetric random walk $S$}. The following Lemma gives the convexity of its free energy w.r.t. the Gaussian disorder $\omega$. 
		
		\begin{lemma}\label{lemma: convexity of Free energy}
			For any fixed $m,\ n\in \mathbb Z_+$, $x,\ y\in \mathbb Z^d$, $\beta >0$, $w\mapsto \log \mathcal Z_{m,n}^{\beta,w }(x,y) $ is a convex function for $w: \mathbb Z^d \to \mathbb R$. The same convexity holds for 
			$
			\log \sum_{i,j=1}^L c_{i,j} \mathcal Z_{m,n}^{\beta,\cdot }(x_i,y_j) 
			$
			with finite $L\in \mathbb Z_+$ and positive $\{ c_{i,j}\}$. 
		\end{lemma}
		\begin{proof}
			For any $w_1,\ w_2: \mathbb Z^d \to \mathbb R$ and $\lambda \in (0,1)$, we can rewrite the partition function \eqref{eqDPRE: Z_{m,n}^{beta,omega}, preliminary} as the linear combination of random walk path: 
			\[
			\begin{aligned}
				&\mathcal Z_{m,n}^{\beta,\lambda w_1 +(1-\lambda) w_2} (x,y) \\
				=& 
				\sum_{\ell _{m,n} \subset ([m,n]\cap \mathbb Z) \times \mathbb Z^2}
				P_{(m,x)}\left(
				S_t=\ell_{m,n}(t),\ t= m,...,n;\ S_n=y
				\right)\\
				&\exp \left\{
				\sum_{t=m+1}^{n-1} \beta (\lambda w_1+ (1-\lambda) w_2 )(t,\ell_{m,n} (t) ) -\frac{\beta^2}{2}
				\right\}
			\end{aligned}
			\]
			where $P_{(m,x)}$ is the law for random walk starting from $(m,x)$. By H\"older inequality, we get $\mathcal Z_{m,n}^{\beta,\lambda w_1 +(1-\lambda) w_2} (x,y) \le \mathcal Z_{m,n}^{\beta,w_1} (x,y)^{\lambda} \mathcal Z_{m,n}^{\beta, w_2} (x,y)^{1-\lambda } $ and the convexity of $\log \mathcal Z_{m,n}^{\beta,\cdot} (x,y)$ follows. The discussion for the linear combination with non-negative coefficients is the same and thus omitted. 
		\end{proof}
		
		By Lemma \ref{lemma: convexity of Free energy}, the free energy of DPRE is a convex functional of the Gaussian disorder and thus Ehrhard-type random variable by Corollary \ref{corollary: Ehrhard for convex functional}. Thus the moment comparison inequality in Theorem  
		\ref{thm: DPRE, moment comparison} is an immediate result of Theorem \ref{thm: Ehrhard moment comparison ineq}.
		
			We now turn to the critical 2D SHF as the weak limit of the two-dimensional DPRE at criticality. Recall the tuning \eqref{eqDPRE: beta_N, critical}, the rescaling \eqref{eq:rescaled-partition-measure}, and the cell averages \eqref{eqDPRE: cell-averaged test}. Thus the marginal partition function recalled from \eqref{eqDPRE: Z_N(f)} is
			\begin{equation}\label{eqPre: Z_N(f)}
				\begin{aligned}
					\mathcal Z_{\llbracket2Nt\rrbracket}^{\beta_N(\theta)}(f)
					&=\frac2N\sum_{z\in\mathbb Z^2_{\mathrm{even}}}f_N(z)\,
					\mathrm E_z\!\left[
					\exp\!\left\{\sum_{n=1}^{\llbracket2Nt\rrbracket-1}
					\left(\beta_N(\theta)\omega(n,S_n)
					-\frac{\beta_N(\theta)^2}{2}\right)\right\}\right]\\
					&=\mathcal Z_{N;0,t}^{\theta}(f,\mathbf1).
				\end{aligned}
			\end{equation}
			For each fixed $t>0$, this converges in law to $\mathscr Z_{0,t}^\theta(f,\mathbf1)$, which is denoted by $\mathscr Z_t^\theta(f)$. 
		
		The following lemma follows from Lemma \ref{lemma: convexity of Free energy}.
		\begin{lemma}
			Fix $t>0$ and $\theta \in \mathbb R$. {Take an arbitrary} non-negative $f\in C_c(\mathbb R^2)$ or $f= c \mathds{1}_A$ for some compact $A\subset \mathbb R^2$. { and $N\in \mathbb Z_+$,} {For any $N\in \bbN$, } $X_N:= \log \cZ_{\llbracket2Nt\rrbracket}^{\beta_N(\theta)} (f)$ is a convex function of some finite-dimensional Gaussian vectors satisfying $\mathbb E [e^{\lambda X_N} ]<\infty$ for every $\lambda >0$. 
		\end{lemma}

		We cite the following weak convergence result from \cite{CSZ2023_2dcSHF} and moment convergence from \cite{CSZ2019_Dickman} \cite{CSZ2019_thirdmoment}. 
		\begin{theorem}
			For any non-negative and integrable $f\in C(\mathbb R^2)$ or $f=c \mathds{1}_A$ for some compact $A\subset \mathbb R^2$, $\cZ_{\llbracket2Nt\rrbracket}^{\beta_N(\theta)} (f)$ weakly converges to $\mathscr{Z}_t^\theta (f)$. For bounded $f$, we have the moments convergence:
			\begin{equation}\label{eqSHF: moments limit from Z_N, h =1,2}
				\mathbb E [\cZ_{\llbracket2Nt\rrbracket}^{\beta_N(\theta)} (f)^h] \to \mathbb E [\mathscr Z_t^\theta (f)^h],\quad \text{ for }h=1,2,3.
			\end{equation}
			Especially, $\mathbb E [\mathscr Z_t^\theta (f)]= \| f\|_{L^1(\mathbb R^2)}$. 
		\end{theorem}
		
			Next, we record the exact forms of the first and second moments. These were first derived in \cite{CSZ2023_2dcSHF}. For the rescaled formulation used in this paper, we instead refer to \cite[(1.1), (1.3)--(1.4)]{ClarkMian2026_SHFflow}. For $t_1<t_2$, we have
			\begin{equation}\label{eqSHF: first moment}
				\mathbb E\big[\mathscr Z_{t_1,t_2}^\theta(\td x,\td y)\big]
				=g_{t_2-t_1}(y-x)\,\td x\,\td y
			\end{equation}
			and
			\begin{equation}\label{eqSHF: second moment formula}
				\begin{aligned}
					&\cov\!\left[
					\mathscr Z_{t_1,t_2}^\theta(\td x,\td y),
					\mathscr Z_{t_1,t_2}^\theta(\td x',\td y')\right]\\
					&\quad=4\pi\,g_{(t_2-t_1)/2}\!\left(\frac{y+y'}2-\frac{x+x'}2\right)
					\int_{t_1<a<b<t_2}
					g_{2(a-t_1)}(x'-x)\,G_\theta(b-a)\\
					&\hspace{48mm}\times g_{2(t_2-b)}(y'-y)\,\td a\,\td b
					\;\td x\,\td y\,\td x'\,\td y'.
				\end{aligned}
			\end{equation}
			where
			\begin{equation}\label{eqSHF: G definition}
				G_\theta(t):=\int_0^\infty\frac{e^{\theta r}t^{r-1}}{\Gamma(r)}\,\td r
				=\int_0^\infty\frac{e^{\theta r}r t^{r-1}}{\Gamma(r+1)}\,\td r,
				\qquad t>0.
			\end{equation}
			Here $\Gamma$ denotes the Gamma function. {We take the same notation $G_\theta$ here as in \cite{CSZ2023_2dcSHF}, but we mention that the they are slightly different because of the rescaling we used in this paper.}
			
			In the notation of \cite{GQT_2021_moments} and  \cite[(2.4)]{Tsai2024_SHF_by_moments}, this is exactly $\mathsf j^\theta(t)$.
			
			\begin{remark}\label{rem:G-normalization-CSZ-LZ}
				Write $G^{\mathrm{CSZ}}_\vartheta$ for the counterpart of \eqref{eqSHF: G definition} in \cite[(3.39)]{CSZ2023_2dcSHF} and \cite[(2.4)]{LZ2024_micromomentSHF}:
				\[
				G^{\mathrm{CSZ}}_\vartheta(t)
				=\int_0^\infty\frac{e^{(\vartheta-\gamma_{\mathrm E})r} r t^{r-1}}{\Gamma(r+1)}\,\td r.
				\]
				The function used throughout the present paper satisfies
				\begin{equation}\label{eqSHF: G comparison}
					G_\theta(t)
					=G^{\mathrm{CSZ}}_{\theta+\gamma_{\mathrm E}}(t)
					=2G^{\mathrm{CSZ}}_{\theta+\gamma_{\mathrm E}-\log2}(2t).
				\end{equation}
				
			\end{remark}
			
			Integrating \eqref{eqSHF: second moment formula} in the terminal variables gives, for bounded integrable $f,h$,
			\begin{equation}\label{eqSHF: marginal covariance}
				\begin{aligned}
					&\cov\!\left[
					\mathscr Z_{t_1,t_2}^\theta(f,\mathbf1),
					\mathscr Z_{t_1,t_2}^\theta(h,\mathbf1)\right]\\
					&\quad=4\pi\int_{t_1<a<b<t_2}G_\theta(b-a)
					\int_{\mathbb R^2}(g_{a-t_1}*f)(z)(g_{a-t_1}*h)(z)\,\td z\,\td a\,\td b.
				\end{aligned}
			\end{equation}
			In particular, for $r>0$,
			\begin{equation}\label{eqSHF: Gaussian second moment}
				\begin{aligned}
					\mathbb E\!\left[\mathscr Z_{t_1,t_2}^\theta(g_r,\mathbf1)^2\right]
					&=1+\int_{0<a<b< t_2-t_1 }\frac{G_\theta(b-a)}{r+a}\,\td a\,\td b\\
					&=1+\int_0^{ t_2-t_1 } G_\theta(v)
					\log\!\left(\frac{r+\frac{t_2-t_1}{2}-v}{r}\right)\,\td v.
				\end{aligned}
			\end{equation}
			We use the following estimates, obtained from \cite[Propositions~2.1--2.2]{LZ2024_micromomentSHF} after the parameter substitution in \eqref{eqSHF: G comparison}. If $q>e^\theta$, then
			\begin{equation}\label{eq:intmom-G-Laplace}
				\int_0^\infty e^{-qt}G_\theta(t)\,\td t
				=\frac{1}{\log q-\theta}.
			\end{equation}
			As $t\downarrow0$,
			\begin{equation}\label{eqSHF: G_theta (t small)}
				G_\theta(t)
				=\frac{1}{t(\log t^{-1})^2}
				\left(1+\frac{2(\theta+\gamma_{\mathrm E})}{\log t^{-1}}
				+O\!\left(\frac{1}{(\log t^{-1})^2}\right)\right).
			\end{equation}
		In particular, there exist $t_0\in(0,e^{-4})$ and $C_G<\infty$ such that
		\begin{equation}\label{eqSHF: G_theta upper}
			G_\theta(t)
			\leq \frac{1}{t(\log t^{-1})^2}
			+\frac{C_G}{t(\log t^{-1})^3},
			\qquad 0<t\leq t_0.
		\end{equation}
		We mention that $t_0<e^{-4}$ will be used later in the proof of Lemma \ref{lem:intmom-F-bound}. 
		
		Fix $\lambda_0>e^\theta$. For
		$\lambda\geq\lambda_0$ and $w>0$, define the weighted Laplace transform
		\begin{equation}\label{eqPre: F_lambda(omega)} 
			F_\lambda(w)
			:=\int_0^\infty
			\frac{e^{-\lambda v}G_\theta(v)}{w+v/2}\,\mathrm{d}v.
		\end{equation}
		{$F_\gl$} is positive and nonincreasing. By $\frac{1}{w+v/2} = \int_0^\infty e^{-\sigma(w+v/2)} \td \sigma$ and  \eqref{eq:intmom-G-Laplace}, we have
		\begin{equation}\label{eqPre: F_lambda 2}
			F_\lambda(w)
			=\int_0^\infty
			\frac{e^{-\sigma w}}
			{\log(\lambda+\sigma/2)-\theta}
			\,\mathrm{d}\sigma.
		\end{equation}
		
		Moreover, for every $w>0$,
		\begin{equation}\label{eq:intmom-F-regular}
			\begin{aligned}
				F_\lambda(w)
				\leq&
				\frac{1}{\log\lambda-\theta}
				\int_0^\infty e^{-\sigma w}\,\mathrm{d}\sigma
				= 
				\frac{1}{w(\log\lambda-\theta)}.
			\end{aligned}
		\end{equation}

		\begin{lemma}\label{lem:intmom-F-bound}
			Fix $\lambda_0>e^{\theta}$. There exist $w_0>0$ small enough and $C_F<\infty$, depending on $\theta$ and
			$\lambda_0$, such that, uniformly in $\lambda\geq\lambda_0$,
			\begin{equation}\label{eq:intmom-F-sharp}
				F_\lambda(w)
				\leq \frac{1}{w\log\frac{1}{w}}
				+\frac{C_F}{w(\log\frac{1}{w})^2},
				\qquad 0<w\leq w_0.
			\end{equation}
			
		\end{lemma}
		
		\begin{proof}
			Recall $t_0\in (0,e^{-4})$ is introduced to ensure the estimate \eqref{eqSHF: G_theta upper}.
			Use the definition \eqref{eqPre: F_lambda(omega)} and split the $v$-integral
			into $(0,w)$, $(w,t_0)$, and $(t_0,\infty)$.  On $(0,w)$,
			\eqref{eqSHF: G_theta upper} gives
			\begin{align}
				\int_0^w\frac{e^{-\lambda v}G_\theta(v)}{w+v/2}\,\mathrm{d}v
				&\leq\frac1w\int_0^w
				\left(\frac{1}{v(\log v^{-1})^2}
				+\frac{C_G}{v(\log v^{-1})^3}\right)\,\mathrm{d}v \notag\\
				&\leq\frac{1}{w\log\frac{1}{w}}
				+\frac{C}{w(\log\frac{1}{w})^2}.
				\label{eq:intmom-F-first}
			\end{align}
			On $(w,t_0)$, we have
			\begin{equation}\label{eq:intmom-F-second}
				\int_w^{t_0}\frac{e^{-\lambda v}G_\theta(v)}{w+v/2}\,\mathrm{d}v
				\leq C\int_w^{t_0}
				\left(\frac{1}{v^2(\log v^{-1})^2}
				+\frac{1}{v^2(\log v^{-1})^3}\right)\,\mathrm{d}v
				\leq\frac{C}{w(\log\frac{1}{w})^2}
			\end{equation}
			{
				where the second inequality can be obtained from integration by parts. Indeed, 
				set $I:=\int_w ^{t_0} v^{-2} (\log v^{-1}) ^{-2}$, then 
				\[
				\begin{aligned}
					I = &\frac{1}{w (\log w^{-1}) ^2} - 
					\frac{ 1}{t_0 (\log t_0^{-1}) ^2} + 2 \int _w ^{t_0} \frac{\dd v}{v^2 (\log v^{-1}) ^3}\\
					\le&  \frac{1}{w (\log w^{-1}) ^2}  + \frac{2}{\log t_0^{-1} } \int_w^{t_0} \frac{\dd v}{v^2 (\log v^{-1}) ^2}\\
					=& \frac{1}{w (\log w^{-1}) ^2}  + \frac{2}{\log t_0^{-1} } I
				\end{aligned}
				\]
				By our assumption $t_0<e^{-4}$, we get $2/\log t_0^{-1} \in (0,1)$, and thus $I \le \frac{C}{w(\log\frac{1}{w})^2}$ for some $C>0$. A similar method holds for $\int_w ^{t_0} v^{-2} (\log v^{-1}) ^{-3}$ as well. 
			}
				On $[t_0,\infty)$, by \eqref{eq:intmom-G-Laplace} and $\lambda\ge\lambda_0>e^\theta$, we have
				\[
				\begin{aligned}
					\int_{t_0}^\infty
					\frac{e^{-\lambda v}G_\theta(v)}{w+v/2}\,\td v
					&\le\frac2{t_0}\int_{t_0}^\infty e^{-\lambda_0v}G_\theta(v)\,\td v\\
					&\le\frac{2}{t_0(\log\lambda_0-\theta)}<C
				\end{aligned}
				\]
			{for some $C>0$, uniformly over $w\in (0,t_0]$ and $\gl\ge \gl_0$. }After decreasing $w_0$, this bounded contribution is absorbed by the second term in \eqref{eq:intmom-F-sharp}.  This proves
			\eqref{eq:intmom-F-sharp}.  
		\end{proof}
		
		We also record the local second-moment estimate. It follows from \eqref{eqSHF: Gaussian second moment} and the corresponding ball-average estimate (see for example \cite{CSZ2019_thirdmoment}): for every fixed $t>0$ and $\theta\in\mathbb R$, there exists $C:=C(t,\theta)\ge 1$ such that,
		\begin{equation}\label{eqSHF: local second moments}
			C^{-1} \log \frac{1}{\epsilon }  	\le \mathbb E [ {\mathscr Z_t^\theta} (g_\epsilon ) ^2 ]\le C \log \frac{1}{\epsilon }  
		\end{equation}
		for all $\epsilon$ small enough. The same estimate holds when we replace $g_\epsilon$ by $\epsilon^{-2} \mathds{1}_{B(0,\epsilon)}$. 
		
		To end this section, we recall the flow property of the SHF proved in \cite{ClarkMian2026_SHFflow}. It will only be used in Subsection \ref{subsection: Local negative moments lower bound}. For fixed $t_1<t_2<t_3$ and $\epsilon>0$, define
		\begin{equation}\label{eq:flow-property}
			\mathscr Z_{t_1,t_2}^\theta
			\bullet_\epsilon
			\mathscr Z_{t_2,t_3}^\theta (\td x,\td y)
			:=
			\int_{\mathbb R^2\times\mathbb R^2}
			\mathscr Z_{t_1,t_2}^\theta(\td x,\td z_1)
			g_\epsilon(z_1-z_2)
			\mathscr Z_{t_2,t_3}^\theta(\td z_2,\td y).
		\end{equation}
		By \cite[Proposition 2.6]{ClarkMian2026_SHFflow},
		$
		\mathscr Z_{t_1,t_2}^\theta
		\bullet_\epsilon
		\mathscr Z_{t_2,t_3}^\theta
		\longrightarrow
		\mathscr Z_{t_1,t_3}^\theta=: \mathscr Z_{t_1,t_2}^\theta
		\bullet
		\mathscr Z_{t_2,t_3}^\theta
		$
		vaguely in $L^2(\mathbb P)$ as $\epsilon\downarrow0$. 
		Moreover, by \cite[Remark 2.8]{ClarkMian2026_SHFflow},
		the convergence holds in $L^2(\mathbb P)$ after integration against a larger class of test functions with suitable exponential decay. 
		In particular, for every fixed $r>0$,
		\[
		\big(\mathscr Z_{t_1,t_2}^\theta\bullet_\epsilon
		\mathscr Z_{t_2,t_3}^\theta\big)(g_r,\mathbf1)
		\longrightarrow\mathscr Z_{t_1,t_3}^\theta(g_r,\mathbf1)
		\quad\text{in }L^2(\mathbb P),\qquad \epsilon\downarrow0.
		\]
		
		Iterating the above relation, $\mathscr Z_{s,t}^\theta$ can be reconstructed from the time-blocks
		\begin{equation}\label{eq: flow, indepedent segments}
			\mathscr Z_{t_0,t_1}^\theta,...,
			\mathscr Z_{t_{n-1},t_n}^\theta
		\end{equation}
		for any partition
		$s=t_0<t_1<\cdots<t_n=t$, i.e. $\mathscr Z_{s,t}^\theta = \mathscr Z_{t_0,t_1}^\theta \bullet ... \bullet \mathscr Z_{t_{n-1},t_n}^\theta$. The random measures in \eqref{eq: flow, indepedent segments} are mutually independent \cite[Proposition 2.2(iii)]{ClarkMian2026_SHFflow}. The flow property also yields the following martingale property: 
		\begin{lemma}\label{lemma: MG of SHF}
			Fix $t_1>0$ and $\theta\in\mathbb R$, and let
			$
			\mathcal F_t:=\mathscr F_{[0,t]},$ with $t\ge t_1,
			$ 
			be the natural filtration of the SHF $\mathscr Z^\theta$. Then, for every nonnegative $f\in L^1(\mathbb R^2)\cap C(\mathbb R^2 )$,
			\[
			M_t:=\mathscr Z_{t_1,t}^\theta(f,\mathbf 1),
			\qquad t>t_1,
			\]
			is an $(\mathcal F_t)$-martingale.
		\end{lemma}
		
		\begin{proof}
			Take a family of cutoff functions
			\[
			\chi_R\in C_c^\infty (\mathbb R^2),\quad \chi_R\ge 0  
			\]
			satisfying $\chi_R \uparrow 1$. Fix $T>t$.
			By the flow property of $\mathscr Z^\theta$, we have 
			\begin{equation}\label{eqPrep: martingale, mid 1}
				\mathbb E \left[
				\mathscr Z_{t_1,T}^\theta (f\chi_R , \chi_R)
				| \mathcal F_t
				\right]= \mathscr Z_{t_1,t}^\theta (f\chi_R ,g_{T-t} * \chi_R ).
			\end{equation}
			Since 
			\[
			\begin{aligned}
				&\mathbb E \left |
				\mathscr Z_{t_1,t}^\theta (f\chi_R ,g_{T-t} * \chi_R ) 
				- \mathscr Z _{t_1,t} ^\theta (f,\mathbf 1)
				\right | \\
				\le& 
				\mathbb E \left[
				\mathscr Z_{t_1,t}^\theta (f \chi_R, | g_{T-t} * \chi_R- \mathbf 1|)
				\right] 
				+ \mathbb E \left[
				\mathscr Z_{t_1,t} ^\theta (|f\chi_R-f|, \mathbf 1)
				\right],
			\end{aligned}
			\]
			by the first moment formula \eqref{eqSHF: first moment} and dominated convergence theorem, we have 
			\[
			\lim_{R\to \infty }\mathbb E \left |
			\mathscr Z_{t_1,t}^\theta (f\chi_R ,g_{T-t} * \chi_R ) 
			- \mathscr Z _{t_1,t} ^\theta (f,\mathbf 1)
			\right | =0. 
			\]
			Similar discussion gives 
			\[
			\lim_{R\to \infty } 
			\mathbb E \left | 
			\mathscr Z_{t_1,T}^\theta (f\chi_R , \chi_R)-
			\mathscr Z_{t_1,T}^\theta (f , \mathbf 1)
			\right | =0. 
			\]
			Taking $R\to \infty$ in \eqref{eqPrep: martingale, mid 1} and using the two $L^1(\mathbb P)$ convergences above, we have 
			\[
			\mathbb E \left[
			\mathscr Z_{t_1,T}^\theta (f,\mathbf 1)| \mathcal F_t
			\right]= \mathscr Z^\theta _{t_1,t} (f , \mathbf 1)
			\]
			which finishes the proof. 
		\end{proof}
		
		\section{Moment comparison inequalities of SHF}
		\label{section: negative moments proof}
		
		In this section, we prove Theorem \ref{thm: SHF moment comparison ineq} and Theorem \ref{thm: left tail of SHF} which concern the moments comparison inequality and left tail of SHF, respectively. 
		\begin{proof}[Proof of Theorem \ref{thm: SHF moment comparison ineq}]
			Set $M_N :=\cZ_{\llbracket2Nt\rrbracket}^{\beta_N(\theta)} (f) $ and $M:= \mathscr Z_t^\theta (f)$. By Lemma \ref{lemma: convexity of Free energy}, $\log M_N$ is a convex functional of some finite-dimensional Gaussian vector and thus within the setting of Theorem \ref{thm: Chen, thm 1 on convexity}.
			
			We start with showing \eqref{eqSHF: p<0, 1<p<2 moments upper bounds}. By Theorem \ref{thm: Chen, thm 1 on convexity} and Lemma \ref{thm: Ehrhard moment comparison ineq} with $p =2$, we have 
			\[
			\mathbb E [M_N^p] \le  (\mathbb E [M_N])^{p(2-p )}
			(\mathbb E [M_N^2 ])^{\frac{p(p-1)}{2} },\quad \forall p\in (-\infty ,0) \cup [1,2].  
			\]
			By \eqref{eqSHF: moments limit from Z_N, h =1,2}, we have
			\begin{equation}\label{eq: E Z_N ^p <, p<0}
				\limsup_N \mathbb E [ \cZ_{\llbracket2Nt\rrbracket}^{\beta_N(\theta)} (f) ^p ] 
				\le 
				\left( \mathbb E [ \mathscr Z_t^\theta (f) ]\right)^{p(2-p)}
				\left( \mathbb E [ \mathscr Z_t^\theta (f) ^2]\right)^{\frac{p(p-1)}{2}},\quad \forall p\in (-\infty ,0) \cup (1,2].
			\end{equation}
			Since $\cZ_{\llbracket2Nt\rrbracket}^{\beta_N(\theta)} (f)$ converges weakly to $\sZ_t^\theta (f)$, we have, for any $K >0$,
			\begin{equation}\label{eqSHF: a truncation argument}
				\mathbb E [ \sZ_{t}^{\theta} (f) ^p\wedge K] = \limsup_N \mathbb E [ \mathcal Z_{\llbracket2Nt\rrbracket}^{\beta_N(\theta)} (f) ^p\wedge K] \le 	\left( \mathbb E [ \sZ_t^\theta (f) ]\right)^{p(2-p)}
				\left( \mathbb E [ \mathscr Z_t^\theta (f) ^2]\right)^{\frac{p(p-1)}{2}}.
			\end{equation}
			Letting $K \uparrow \infty$ proves \eqref{eqSHF: p<0, 1<p<2 moments upper bounds}. 
			
			From the boundedness of negative moments given in \eqref{eqSHF: p<0, 1<p<2 moments upper bounds}, we have $\mathscr Z_t^\theta (f)$ is a.s. positive. Since $\cZ_{\llbracket2Nt\rrbracket}^{\beta_N(\theta)} (f)$ converges weakly to $\mathscr Z_t^\theta (f)$, $\log \cZ_{\llbracket2Nt\rrbracket}^{\beta_N(\theta)} (f)$ also converges weakly to $\log \mathscr Z_t^\theta (f)$. By Lemma \ref{lemma: Ehrhard type is close}, $\log \mathscr Z_t^\theta (f)$ is Ehrhard-type and thus the moment comparison inequality \eqref{eqSHF: moment comparison inequality} and \eqref{eqSHF: E log Z two-side bound} follow from Theorem \ref{thm: Ehrhard-type convexity} and Lemma \ref{thm: Ehrhard moment comparison ineq}.   
		\end{proof}

		\begin{proof}[Proof of Theorem \ref{thm: left tail of SHF}]
			From the proof of Theorem \ref{thm: SHF moment comparison ineq} we know that $\log \sZ^\gt_t(f)$ is Ehrhard-type. By Lemma \ref{lemma: Ehrhard and convex Gaussian}, it is then a convex functional of a Gaussian random variable. Hence we can conclude from Theorem \ref{thm: Chen, theorem 3, left tail}.
		\end{proof}
		
		\section{Moments estimate for locally averaged SHF}\label{section: local moments of SHF}
		In this section we prove Theorem \ref{thm: SHF, local moments}  concerning the moments of locally averaged SHF. {{As we will see in the proof, the moment comparison inequality \eqref{eqSHF: moment comparison inequality} provides a series of one-direction inequalities and a powerful framework. However, to get the moment bounds in the other direction, we need two additional inputs (Proposition \ref{propo: SHF integer moment upper bound} and Proposition \ref{propo: SHF local moments, p<0, lower bound}) from the specific model--SHF.} }
		
		Recall that we write $\mathcal U_\epsilon (x):= \frac{1}{\pi \epsilon^2} \mathds{1}_{B(0,\epsilon)}$. 
		
		\begin{proof}[Proof of Theorem \ref{thm: SHF, local moments}]
			
			The cases when $p\in \{ 0,1,2\}$ are trivial so we omit them here. By the moment comparison inequality \eqref{eqSHF: moment comparison inequality}, we have 
			\begin{enumerate}
				\item When $p\in (-\infty ,0) \cup (1,2)$,  
				$
				\mathbb E [\mathscr{Z}_t^\theta (\mathcal U_\epsilon )^p] \le (\mathbb E [ \mathscr{Z}_t^\theta (\mathcal U_\epsilon)^2] )^{p(p-1)/2}.
				$
				\item When $p\in (0,1)$,
				$
				\mathbb E [\mathscr{Z}_t^\theta (\mathcal U_\epsilon)^p] \ge (\mathbb E [ \mathscr{Z}_t^\theta (\mathcal U_\epsilon)^2] )^{p(p-1)/2}.
				$
				\item When $q\ge 2$,
				$
				\mathbb E [\mathscr{Z}_t^\theta (\mathcal U_\epsilon)^q] \ge  (\mathbb E [ \mathscr{Z}_t^\theta (\mathcal U_\epsilon)^2] )^{q(q-1)/2}.
				$
			\end{enumerate}
			It remains to estimate the reverse direction. We introduce here a basic comparison between SHF tested on the small ball and heat kernel which will be used later: Since
			\[
			\mathcal U_\epsilon (x) =\frac{1}{\pi \epsilon ^2 } \mathds{1}_{B(0,\epsilon)} (x) \le e \cdot g_{\epsilon^2 /2 } (x),\quad \forall x\in \mathbb R^2,
			\]
			and SHF as a random measure satisfies $\mathscr Z_t^\theta (a f) = a \mathscr Z_t^\theta (f)$ for any $a>0$ and $f\in L^1 (\mathbb R^2)$ (refer to \cite{CSZ2023_2dcSHF}), we have $\mathscr Z_t^\theta (\mathcal U_\epsilon ) \le C \mathscr Z_t^\theta ( g_{\epsilon ^2 /2}).$
			
			We start with the upper bound for each $p\ge 2$ moment. By Proposition \ref{propo: SHF integer moment upper bound}, for every integer $h\ge 2$,
			\[
			\mathbb E [\mathscr Z_t^\theta (\mathcal U_\epsilon) ^h] 
			\le C \mathbb E [\mathscr Z_t^\theta (g_{\epsilon ^2 /2 }) ^h] 
			\le C 
			\left(
			\log \frac{1}{\epsilon}
			\right)^{\frac{h(h-1)}{2} }
			\]
			As a result of the moment comparison \eqref{eqSHF: moment comparison inequality} with $p\ge 2$ and integer $\bbN \ni h >p$, we have 
			\[
			\mathbb E [\mathscr Z_t^\theta ( \mathcal U_\epsilon   ) ^p] 
			\le 
			C \left(
			\mathbb E [\mathscr Z_t^\theta ( \mathcal U_\epsilon    ) ^h]
			\right)^{\frac{p(p-1)}{h(h-1) }}
			\le C \left(
			\log \frac{1}{\epsilon}
			\right)^{\frac{p(p-1)}{2} }.
			\]
			By \eqref{eqSHF: local second moments}, we further have 
			\[
			\mathbb E [\mathscr Z_t^\theta ( \mathcal U_\epsilon   ) ^p] 
			\le C 
			\left(
			\mathbb E [\mathscr Z_t^\theta ( \mathcal U_\epsilon   )^ 2] 
			\right)^{\frac{p(p-1)}{2}}.
			\] 
			
			We now consider the other cases. By Proposition \ref{propo: SHF local moments, p<0, lower bound}, for every $h<0$, there is $C >0$ such that,
			\begin{equation}\label{eqLocal: mid 1}
				\mathbb E [\mathscr Z_t^\theta ( \mathcal U_\epsilon   ) ^h] \ge 
				C \mathbb E [\mathscr Z_t^\theta ( g_{\epsilon ^2 /2 }   ) ^h] \ge  
				C  \left(
				\log \frac{1}{\epsilon}
				\right)^{\frac{h(h-1)}{2} }.
			\end{equation}
			Furthermore, by the second moment estimate \eqref{eqSHF: local second moments}, we have 
			\[
			\mathbb E [\mathscr Z_t^\theta ( \mathcal U_\epsilon   ) ^h] \ge  C \left(
			\mathbb E [\mathscr Z_t^\theta ( \mathcal U_\epsilon   ) ^2] 
			\right)^{\frac{h(h-1)}{2}}.
			\]
			From the comparison inequality \eqref{eqSHF: moment comparison inequality}, we have, for any fixed $h<0$,  
			\begin{equation}\label{eqLocal: mid 2}
				\mathbb E [\mathscr Z_t^\theta ( \mathcal U_\epsilon) ^p]  \ge 
				\left(
				\mathbb E [\mathscr Z_t^\theta ( \mathcal U_\epsilon) ^h]  
				\right)^{\frac{p(p-1)}{h(h-1)}},\quad \forall p\in (1,2)
			\end{equation}
			and 
			\begin{equation}\label{eqLocal: mid 3}
				\mathbb E [\mathscr Z_t^\theta ( \mathcal U_\epsilon) ^p]  \le 
				\left(
				\mathbb E [\mathscr Z_t^\theta ( \mathcal U_\epsilon) ^h]  
				\right)^{\frac{p(p-1)}{h(h-1)}},\quad \forall p\in (0,1).
			\end{equation}
			Combining \eqref{eqLocal: mid 1}-\eqref{eqLocal: mid 3} finishes the proof.  
		\end{proof}

		\subsection{Upper bound for the integer moments of locally averaged SHF}
		\label{subsection: int local moments, upper bound}
		
		In this subsection, we study the upper bound for local integer moments of SHF. We take notations 
		\[
		\mathcal M_{h,\epsilon}^{\theta}
		:=\mathbb E\!\left[\left(
		\mathscr   Z_1^\theta(g_{\epsilon^2/2})\right)^h\right].
		\]
		For an integer $h\ge 2$, write
		\begin{equation}\label{eqU: q_h}
			q_h:=\frac{h(h-1)}{2}-1.
		\end{equation}
		
		\begin{proposition}\label{propo: SHF integer moment upper bound}
			For every fixed integer $h\ge2$ and every $\theta\in\mathbb R$, there are constants $C=C(h,\theta)<\infty$ and $\varepsilon_0=\varepsilon_0(h,\theta)>0$
			such that
			\begin{equation}\label{eq:intmom-main}
				\mathcal M_{h,\varepsilon}^{\theta}
				\leq C\left(\log\frac1\varepsilon\right)^{\binom{h}{2}},
				\qquad 0<\varepsilon\leq\varepsilon_0.
			\end{equation}
			
		\end{proposition}

		\begin{proof}
				For $h=2$, the result follows from \eqref{eqSHF: local second moments}. Assume henceforth that $h\ge3$.
			We prove the proposition in two steps. 
			
			\textbf{Step 1: Linear-operator representation of the upper bounds.}
				Let $G_\theta$ be defined by \eqref{eqSHF: G definition}. We adapt \cite[Lemma~3.2]{LZ2024_micromomentSHF} in the present rescaled convention; for the difference of different conventions we refer to Remark~\ref{rem:SHF-normalization-CSZ}. This yields
				\begin{equation}\label{eq:intmom-LZ-input}
					\begin{aligned}
						\mathcal M_{h,\varepsilon}^{\theta}
						\le & C_{h,\theta}\left(1+\log\frac1\varepsilon\right)\\
						&+\sum_{m=2}^{\infty}\binom{h}{2}q_h^{m-1}
						\int_{\mathcal D_{m,\varepsilon}}
						\frac{1}{u_1}
						\left(\prod_{r=1}^{m-1}
						\frac{G_\theta(v_r)}{\frac12(u_r+v_r)+u_{r+1}}
						\right)G_\theta(v_m)
						\,\mathrm{d}\boldsymbol u\,\mathrm{d}\boldsymbol v,
					\end{aligned}
				\end{equation}
			where
			\begin{equation}\label{eq:intmom-domain} 
				\mathcal D_{m,\varepsilon}
				:=\left\{(\boldsymbol u,\boldsymbol v)\in[0,\infty)^{2m}:
				\sum_{i=1}^m(u_i+v_i)\leq1+\varepsilon^2/2,
				\quad u_1>\varepsilon^2/2\right\}.
			\end{equation}
			Fix $\lambda_0>e^{\theta}$.  For
			$\lambda\geq\lambda_0$, define
			\begin{equation}\label{eq:intmom-F-def}
				F_\lambda(w)
				:=\int_0^\infty
				\frac{e^{-\lambda v}G_\theta(v)}{w+v/2}\,\mathrm{d}v,\quad \forall w>0
			\end{equation}
			which is positive and nonincreasing. On the domain \eqref{eq:intmom-domain},
			$\sum_i(u_i+v_i)\leq2$.  Hence
			\[
			1\leq e^{2\lambda}
			\exp\!\left(-\lambda\sum_{i=1}^m(u_i+v_i)\right).
			\]
			Insert this factor in \eqref{eq:intmom-LZ-input}.  Integrating first in $v_m$ by \eqref{eq:intmom-G-Laplace}, and then successively in
			$v_{m-1},\ldots,v_1$ by \eqref{eqPre: F_lambda 2}, gives
			\begin{equation}\label{eq: M_{h,epsilon}^theta <}
				\mathcal M_{h,\varepsilon}^{\theta}
				\le C_{h,\theta,\lambda_0}e^{2\lambda}
				\left(1+\log\frac1\varepsilon
				+\sum_{m=2}^\infty\mathcal J_{m,\varepsilon}^{(h,\lambda)}\right),
			\end{equation}
			where
			\begin{equation}\label{eq: J_{m,epsilon}^{h,lambda}}
				\begin{aligned}
					\mathcal J_{m,\varepsilon}^{(h,\lambda)}
					:= &\binom{h}{2}q_h^{m-1}
					\int_{\substack{u_i\geq0,\ \sum_{i=1}^m u_i\leq2\\
							u_1>\varepsilon^2/2}}
					\frac{\td u_1}{u_1}
					\prod_{r=2}^m
					F_\lambda\!\left(u_r+\frac{u_{r-1}}2\right)
					\td u_2...\td u_m\\
					\le & 
					\binom{h}{2}q_h^{m-1}
					\int_{\epsilon^2/2} ^2 \frac{\td u_1}{u_1}
					\int_{[0,2]^{m-1}}
					\prod_{r=2}^m
					F_\lambda\!\left(u_r+\frac{u_{r-1}}2\right)
					\td u_2...\td u_m.
				\end{aligned}
			\end{equation}
			For any nonnegative measurable function $f$ on $(0,2]$, define the integral operator
			\begin{equation}\label{eq: K_lambda f}
				(K_\lambda f)(x)
				:=\int_0^2F_\lambda\!\left(y+\frac{x}{2}\right)f(y)\,\td y.
			\end{equation}
			Then \eqref{eq: J_{m,epsilon}^{h,lambda}} is equivalent to
			\[
			\mathcal J_{m,\varepsilon}^{(h,\lambda)}\le 
			\frac{h(h-1)}{2} \int _{\epsilon^2/2} ^2 
			(q_h K_\lambda)^{m-1}\textbf{1}(x) \frac{\td x}{x}
			\]
			where $(q_h K_\lambda)^{m-1}$ means the composition of linear operators. We can then bound 
			\begin{equation}\label{eqU: represtt of upper bounds for M_epsilon}
				\mathcal M_{h,\varepsilon}^{\theta} \le 
				C e^{2\lambda} (1+ \log \frac{1}{\epsilon}) 
				+ C e^{2\lambda} 
				\frac{h(h-1)}{2} \int _{\epsilon^2/2} ^2 
				\sum _{k=0}^\infty (q_h K_\lambda)^{k}\textbf{1}(x) \frac{\td x}{x}
			\end{equation}
			
			\textbf{Step 2: Construction of the super solutions.}
			In Lemma~\ref{lemma: H(x) super solution}, we choose a fixed large $\lambda >0$ and construct a positive function $H:(0,2] \to [1,\infty )$ such that 
			\[
			1+q_h K_\lambda H(x) \le H(x),\quad x\in (0,2].
			\]
			By induction, we have 
			\begin{equation}\nonumber
				\sum_{k=0}^N(q_hK_\lambda)^k\mathbf{1}(x)\leq H(x),
				\qquad N\ge0,\ x\in(0,2].
			\end{equation}
			Letting $N\to\infty$, by monotone convergence theorem, we have
			\begin{equation}\label{eqU: sum J < int H(x)/x}
				\sum_{m=2}^\infty\mathcal J_{m,\varepsilon}^{(h,\lambda)}
				\leq \binom{h}{2}\int_{\varepsilon^2/2}^2\frac{H(x)}{x}\,\mathrm{d}x.
			\end{equation}
			For $0<\varepsilon\leq\sqrt{\delta}$ with $\delta >0$ fixed and shown in Lemma \ref{lemma: H(x) super solution}, the definition \eqref{eqU: H(x)} of $H$ gives
			\begin{equation}\label{eq: int H(x)/x}
				\begin{aligned}
					\int_{\varepsilon^2/2}^2\frac{H(x)}x\,\td x
					&\leq\int_{\varepsilon^2/2}^{\delta}
					\frac{(\log \frac{1}{x})^{q_h}}x\,\td x
					+B\int_\delta^2\frac{\,td x}{x}\\
					&\leq C_{h,\theta}
					\left(\log\frac1\varepsilon\right)^{q_h+1}
					=C_{h,\theta}
					\left(\log\frac1\varepsilon\right)^{\binom{h}{2}}.
				\end{aligned}
			\end{equation}
			
			Plugging \eqref{eq: int H(x)/x} into \eqref{eq: M_{h,epsilon}^theta <} proves
			\eqref{eq:intmom-main}.
		\end{proof}

		\begin{lemma}\label{lemma: H(x) super solution}
			Fix $h\geq3$ and put $q:=q_h\geq2$.  There exist fixed
			$\lambda\geq\lambda_0$, $A>0$ and $\delta>0$, and a function $H:(0,2]\to[1,\infty)$ defined by 
			\begin{equation}\label{eqU: H(x)}
				H(x)=
				\begin{cases}
					(\log \frac{1}{x})^q-A(\log \frac{1}{x})^{q-1},&0<x\leq\delta,\\
					B,&\delta<x\leq2,
				\end{cases}
				\qquad
				B:=(\log\frac{1}{\delta})^q-A(\log\frac{1}{\delta})^{q-1},
			\end{equation}
			such that
			\begin{equation}\label{eq:intmom-supersolution}
				1+qK_\lambda H(x) \le H(x)
				\qquad x\in(0,2].
			\end{equation}
			
		\end{lemma}
		
		\begin{proof}
			We first determine a series of constants that are needed later. Let $C_F$ and $w_0$ be those in Lemma~\ref{lem:intmom-F-bound}.  For the fixed
			integer $q$, choose $C_q>0 $ such that
			\begin{equation}\label{eq:intmom-log-integral}
				\int_0^x(\log \frac{1}{y})^q\,{\td}y
				\leq C_qx(\log \frac{1}{x})^q,
				\qquad 0<x\le e^{-1}.
			\end{equation}
			Set $C_0:=2C_q(1+C_F)$.  Choose $A$ so large that
			\begin{equation}\label{eq:intmom-A-choice}
				A>qC_F+q(q-1)C_0,
			\end{equation}
			and define
			\begin{equation}\label{eq: eta:=, b:=}
				\eta:=\frac{A-qC_F-q(q-1)C_0}{q-1}>0,
				\qquad
				b:=\frac{q(A-C_F)}{q-1}.
			\end{equation}
			Next choose $\delta\in(0,w_0]$ sufficiently small such that:
			\begin{equation}\label{eq:intmom-delta-choice}
				\log\frac{1}{\delta}\geq\max\{2A,b,1\},
				\qquad B\geq2,
				\qquad \eta \big(\log \frac{1}{\delta}\big)^{q-1}\geq\frac32.
			\end{equation}
			Then $L\mapsto L^q-AL^{q-1}$ is increasing on $[\log \frac{1}{\delta},\infty)$, and $H(x)$ in \eqref{eqU: H(x)} satisfies $H\geq B\ge 2$ and
			$H\in L^1(0,2)$. Lastly, put $d_\lambda:=\log\lambda-\theta$.  We choose
			$\lambda\geq\lambda_0$ large enough such that:
			\begin{equation}\label{eq:intmom-lambda-choice}
				\frac{qB(2-\delta)}{\delta d_\lambda}\leq\frac12,
				\qquad
				\frac{2q}{\delta d_\lambda}\int_0^2H(y)\,\mathrm{d}y\leq1.
			\end{equation}
			
			We now verify \eqref{eq:intmom-supersolution}. By \eqref{eq: K_lambda f}, it is equivalent to 
			\begin{equation}\label{eq: 1+qK_lambda H(x) =...<H(x)}
				1+q\int_0^2 F_\lambda(y+ x/2)H(y) \,\mathrm{d}y \le H(x),\quad x\in(0,2]. 
			\end{equation}
			
			We split the proof into the cases $x\in(0,\delta]$ and $x\in(\delta,2]$. First consider $0<x\leq\delta$.  Split the integral 
			over $(0,x)$, $(x,\delta)$, and $(\delta,2)$.
			
			For $0<y<x$, monotonicity of $F_\lambda$, Lemma~\ref{lem:intmom-F-bound},
			and \eqref{eq:intmom-log-integral} imply
			\begin{equation}\label{eq:intmom-small-region}
				\int_0^xF_\lambda\!\left(y+\frac{x}{2}\right)H(y)\,\mathrm{d}y
				\le F_\lambda \big(\frac{x}{2}\big)\int_0^x \big(\log \frac{1}{y}\big)^q \,\mathrm{d}y
				\leq C_0\big(\log \frac{1}{x}\big)^{q-1}.
			\end{equation}
			For $x\leq y\leq\delta$,  since
			$F_\lambda(y+x/2)\leq F_\lambda(y)$, together with Lemma~\ref{lem:intmom-F-bound}, we have 
			\begin{align*}
				F_\lambda(y)H(y)
				&\le \left(\frac{1}{y\log \frac{1}{y}}+\frac{C_F}{y\big(\log \frac{1}{y}\big)^2}\right)
				\left(\big(\log\frac{1}{y}\big)^q-A\big(\log \frac{1}{y}\big)^{q-1}\right)\\
				&\leq\frac{\big(\log \frac{1}{y}\big)^{q-1}-(A-C_F)\big(\log \frac{1}{y}\big)^{q-2}}{y}
			\end{align*}
			Consequently,
			\begin{equation}\label{eq:intmom-middle-region}
				\begin{aligned}
					&q\int_x^\delta
					F_\lambda\!\left(y+\frac{x}{2}\right)H(y)\,\mathrm{d}y  \le\big(\log \frac{1}{x}\big)^q-\big(\log \frac{1}{\delta}\big)^q-b\left(\big(\log \frac{1}{x}\big)^{q-1}-\big(\log \frac{1}{\delta}\big)^{q-1}\right).
				\end{aligned}
			\end{equation}
			where $b$ is defined in \eqref{eq: eta:=, b:=}.
			For $\delta\leq y\leq2$, the first condition in
			\eqref{eq:intmom-lambda-choice} and \eqref{eq:intmom-F-regular} give
			\begin{equation}\label{eq:intmom-large-region}
				q\int_\delta^2
				F_\lambda\!\left(y+\frac{x}{2}\right)H(y)\,\mathrm{d}y
				\le \frac{q}{\delta d_\lambda} \int_\delta^2 B\,\mathrm{d}y
				\leq\frac12.
			\end{equation}
			Combining \eqref{eq:intmom-small-region}--\eqref{eq:intmom-large-region},
			and observing from \eqref{eq: eta:=, b:=} that
			$b-qC_0=A+\eta$, we obtain
			\begin{align*}
				1+qK_\lambda H(x)
				&\leq \big(\log \frac{1}{x}\big)^q-(A+\eta)\big(\log \frac{1}{x}\big)^{q-1}
				+\left(-\big(\log \frac{1}{\delta}\big)^q+b \big(\log \frac{1}{\delta}\big)^{q-1}\right)+\frac32\\
				& \le \big(\log \frac{1}{x}\big)^q -A \big(\log \frac{1}{x}\big)^{q-1}
				-\left(\big(\log \frac{1}{\delta}\big)^q-b \big(\log \frac{1}{\delta}\big)^{q-1}\right)
				-\Big( \eta\big(\log \frac{1}{x}\big)^{q-1} - \frac{3}{2} \Big)
				\\
				&\leq \big(\log \frac{1}{x}\big)^q-A\big(\log \frac{1}{x}\big)^{q-1}=H(x),\quad \forall x\in (0,\delta ].
			\end{align*}
			The last inequality follows from \eqref{eq:intmom-delta-choice}.
			
			We now prove \eqref{eq: 1+qK_lambda H(x) =...<H(x)} for $\delta<x\leq2$. Note that 
			\[
			\begin{aligned}
				1+q\int_0^2 F_\lambda(y+ x/2)H(y) \,\mathrm{d}y \le& 1+q F_\lambda \left( \frac{\delta}{2}\right) \int_0^2 H(y) \,\mathrm{d}y \\
				\le & 1+ \frac{2q}{\delta (\log \lambda -\theta)} \int_0^2 H(y) \,\mathrm{d}y\le 2 
			\end{aligned}
			\]
			where the first inequality is due to the monotonicity of $F_\lambda$ and $y+x/2\ge \delta /2$, the second is due to \eqref{eq:intmom-F-regular}, and the third follows from the second condition in \eqref{eq:intmom-lambda-choice}. Finally,  by \eqref{eqU: H(x)} and \eqref{eq:intmom-delta-choice}, we have $H(x) \equiv B \ge 2$ for any $x\in (\delta, 2]$, which finishes the proof. 
		\end{proof}

		\subsection{Lower bound for the negative moments of locally averaged SHF}\label{subsection: Local negative moments lower bound}
		Recall that we write $\mathscr Z^\theta_{s,t}$ as the two-parameter stochastic heat flow and set
		\[
		\sZ^\theta_{s,t} (f,f')
		:=\int_{\R^4}f(x)f'(y)\sZ^\theta_{s,t}(\dd x,\dd y).
		\]
		In this section, we prove the following lower bound for the local negative moments of SHF.
		
		\begin{proposition}\label{propo: SHF local moments, p<0, lower bound}
			Fix $t>0$, $\theta \in \mathbb R$ and $p<0$. There is $C:= C(p,t,\theta)>0$ and $\epsilon_0\in (0,1/2)$ such that 
			\begin{equation}\label{eqSHF: local p<0 moments lower bound}
				\mathbb E [\sZ_t^\theta(g_{\epsilon^2 },\boldsymbol 1)^p] \ge C \left (\log \frac{1}{\epsilon} \right )^{p(p-1)/2 },\quad \forall \epsilon \in (0,\epsilon_ 0 ). 
			\end{equation} 
			
		\end{proposition}
		
		Without loss of generality, we only discuss the case $t=1$. By the scaling and translational invariance of SHF (refer to Theorem 1.2 in \cite{CSZ2023_2dcSHF}), we have 
		\[
		\mathscr{Z}_{0,t}^\theta (g_{\epsilon^2 }, \textbf{1}) \overset{d}{=} \mathscr Z_{0,1}^{\theta + \log t} (g_{\epsilon^2/t}, \textbf{1} ).
		\]
		And thus we can recover the result for general $t>0$ by proving our estimate holding uniformly for $\theta$ in some interval. Out of technical reasons, we restrict our discussion on SHF with a small time delay like in \cite{GuTsai_2026_loglogCLT}. Again, by the scaling and translational invariance of SHF, we have 
		\[
		\mathscr  Z_{\epsilon^2,1}^\theta (g_{\epsilon^2 },\textbf{1} )
		\overset {d}{=} 
		\mathscr  Z_{0,1}^{\theta + \log (1-\epsilon^2) } 
		(g_{{\epsilon ^2}/{(1-\epsilon ^2 )}} , \textbf{1})
		\]
		Thus, to prove Proposition \ref{propo: SHF local moments, p<0, lower bound}, it suffices to show the following result: 
		\begin{proposition}\label{propo: time delay SHF}
			Fix $c _0>0$ and $p<0$. There exist $C(c_0,p)>0$ and $\epsilon_0 >0$ such that, for all $\theta \in (-c_0,c_0)$, $\epsilon \in (0,\epsilon_0)$,   
			\begin{equation}\label{eqL: E Z_epsilon^2 ,1 ^p >}
				\mathbb E [(\mathscr Z_{\epsilon^2,1}^\theta (g_{\epsilon^2 },\textbf{1} ) )^p] \ge C \left (\log \frac{1}{\epsilon} \right ) ^{\frac{p(p-1)}{2}}
			\end{equation}
		\end{proposition}
		
		We consider the same block expansion as in \cite{GuTsai_2026_loglogCLT}. First, we separate the time interval $[\epsilon^2,1]$ by 
		\[
		t_i:= \epsilon^2 b^{-2i},\quad i=0,...,N:= \lfloor \frac{\log \epsilon^{-1} }{\log b^{-1} } \rfloor 
		\]
		with some $b>0$ small but fixed. Set $p_i:= g_{ {t_i}}$ for each $i$. Note that the choice of starting at $\epsilon ^2 $ instead of $0$ makes this geometric mesh construction more direct. To make it convenient for the discussion here, we may take some $\bar b >0$ small enough and take $b:=b(\epsilon)\in (\frac{\bar b}{2}, \bar b)$ so that $\frac{\log \epsilon^{-1} }{\log b^{-1} } \in \mathbb Z_+$ and thus $t_N =1$. Then, we take an integer $\ell >1$ and set 
		$n:= N-\ell.$ For $1\le i \le n$, define 
		\[
		Z_i:= \sZ_{t_{i-1},t_i} ^\theta,\quad p_i = g_{t_i}.
		\]
		In this subsection, we set 
		$M_i:=   \sZ_{\epsilon^2,t_i}^\theta (g_\epsilon, \textbf{1})$
		for $i=1,...,N$ which is a non-negative martingale by Lemma \ref{lemma: MG of SHF}. As a result, to get a lower bound for negative moments of $M_N = \mathscr Z_{\epsilon^2, 1} (g_\epsilon ,\textbf{1} )$, it suffices to lower bound for negative moments of $M_n$. Also, some estimates are only available when we take $\ell$ large so that $M_n$ covers a short time interval. For example, in the second moment formula \eqref{eqSHF: second moment formula} of SHF, we have the expansion \eqref{eqSHF: G_theta (t small)} for $G_\theta$ only when the time is small. 
		
		By the flow property of SHF (see \eqref{eq:flow-property}), we have 
		$
		M_n=  Z_1 \bullet ...\bullet Z_n(p_0, \textbf{1}).
		$
		In \cite{GuTsai_2026_loglogCLT}, $M_n$ is approximated by the product $\prod_{i=1}^n A_i$, where
		\begin{equation}\label{eqL: A_i}
			A_i:= Z_i (p_{i-1}, \textbf{1}).
		\end{equation}
		The idea of approximating a partition function by a product of contributions from well-separated scales has appeared in several earlier works on marginally relevant polymers and the two-dimensional KPZ equation; see, for instance, \cite{CSZ2020_subcriticalKPZ,LygkonisZygouras2023_ErdosTaylor,CoscoDonadini2025_CLTforDPRE}.
		
		To this end, we let $R_i(\dd x, \dd y):= Z_i( \dd x, \mathbf 1 )\, p_i(y) \dd y$ and define $D_i:= Z_i-R_i$. Proposition 2.1 in \cite{GuTsai_2026_loglogCLT} gives the following decomposition for $M_n$.  
		\begin{proposition}
			For every discrete interval
			\[
			I=\{m,m+1,\ldots,r\}\subseteq \{1,\ldots,n\},
			\qquad |I|\ge 2,
			\]
			define
			\begin{equation}\label{eqL: B_I}
				B_I
				:=
				D_m\bullet D_{m+1}\bullet\cdots\bullet D_{r-1}
				\bullet Z_r(p_{m-1},\mathbf 1)
			\end{equation}
			which has mean zero.
			Let
			\[
			\mathfrak I_n
			:=
			\left\{
			\mathcal I:
			\begin{array}{l}
				\mathcal I \text{ is a nonempty finite collection of pairwise disjoint}\\
				\text{intervals } I\subseteq\{1,\ldots,n\}
				\text{ satisfying } |I|\ge 2
			\end{array}
			\right\},
			\]
			and, for \(\mathcal I\in\mathfrak I_n\), write
			$
			U(\mathcal I):=\bigcup_{I\in\mathcal I} I.
			$
			Then
			\begin{equation}\label{eqL: M_n expansion}
				M_n
				=
				\prod_{i=1}^n A_i
				+
				\sum_{\mathcal I\in\mathfrak I_n}
				\left(\prod_{I\in\mathcal I} B_I\right)
				\left(\prod_{i\notin U(\mathcal I)} A_i\right)
			\end{equation}
			where product over an empty index set is understood to be equal to $1$.
		\end{proposition}
		Let us give a specific example of the expansion above:
		\begin{align*}
			M_4 
			& = 
			\prod_{i=1}^{4}A_i + B_{\{1,2\}}A_3A_4 + 
			B_{\{2,3\}}A_1A_4 + B_{\{3,4\}}A_1A_2 \\ 
			& \quad   + B_{\{1,2,3\}}A_4 + B_{\{2,3,4\}} A_1 
			+ B_{\{1,2\}} B_{\{3,4\}} + B_{\{1,2,3,4\}}.
		\end{align*}
		\begin{proof}
			This is the unnormalised form of \cite{GuTsai_2026_loglogCLT}, Proposition 2.1. We include its proof here for completeness. First, due to $Z_1 = R_1+D_1$, we decompose the first block:
			\begin{align}
				M_n 
				& = (Z_1\bullet \cdots \bullet Z_n)(p_0) \notag \\ 
				& = A_1(Z_2\bullet \cdots \bullet Z_n)(p_1)
				+ (D_1\bullet Z_2 \bullet \cdots \bullet Z_n)(p_0) \label{eq:gt decomp1}
			\end{align}
			Thus in the first term we have decoupled $A_1$ from the rest. For the second term, we use $Z_2 = R_2+D_2$ to get
			\begin{align}
				& (D_1\bullet Z_2 \bullet \cdots \bullet Z_n)(p_0)\notag \\ 
				& =  (D_1\bullet R_2 \bullet Z_3 \bullet \cdots \bullet Z_n)(p_0)
				+ (D_1\bullet D_2 \bullet Z_3 \bullet \cdots \bullet Z_n)(p_0) \notag \\
				& =  (D_1\bullet Z_2)(p_0)  (Z_3 \bullet \cdots \bullet Z_n)(p_2)
				+ (D_1\bullet D_2 \bullet Z_3 \bullet \cdots \bullet Z_n)(p_0).\label{eq:gt decomp2}
			\end{align}
			Thus by repeating the same procedure on $(Z_2\bullet \cdots \bullet Z_n)(p_1)$ in \eqref{eq:gt decomp1} and $(D_1\bullet D_2 \bullet Z_3 \bullet \cdots \bullet Z_n)(p_0)$ in \eqref{eq:gt decomp2} we obtain the result.
		\end{proof}

		Set 
		$V_i := A_i -\mathbb E [A_i]= A_i-1$ 
		as the fluctuation of each $A_i$. As in \cite{GuTsai_2026_loglogCLT}, we consider the events $\Omega_i:= \{|V_i| <1/2\}$ for $i=1,...,N-\ell$ and $\Omega := \bigcap_{i=1}^n \Omega_i$ which make sure the fluctuation of each $A_i$ is not too large. As we will see in the proof later, this restriction enables us to use some basic bounds w.r.t. the negative power function (see for example \eqref{eqL: (1+u)^p >...}). 
		
		The proof of following two Lemmas are postponed to the later part of this section.

		\begin{lemma}\label{lemma: EA_i^p 1_omega_i >...} Fix $p<0$. 
			There exists $\ell_0 >0$ large enough and $C>0$ such that, when $\ell \ge \ell_0$, 
			\begin{equation}\label{eqL: E A_i^p 1_onega_i >}
				\mathbb E [A_i^p \mathds{1}_{\Omega_i}] \ge 
				1+ \frac{p(p-1)/2}{N-i} -\frac{C}{(N-i)^{3/2} }\ge 1,\quad \forall 1\le i \le n= N- \ell . 
			\end{equation}
			As a result, there exists $C(\ell ) >0$ such that, when $\ell \ge \ell_0$, 
			\begin{equation}\label{eqL: E prod A_i Omega}
				\mathbb E [\mathds{1}_{\Omega} (\prod_{i=1}^{n} A_i )^p] \ge C(\ell ) N^{p(p-1)/2}.
			\end{equation}
			
		\end{lemma}
		
		\begin{lemma}\label{lemma: ratio}
			Fix $p<0$. There are $b_0>0$ small and $\ell_0\in \mathbb Z_+$ large such that for all $b\in (0,b_0]$ and $\ell \ge \ell_0$,
			\[
			\left | 
			\mathbb E \left[
			\mathds{1}_{\Omega} (M_n - \prod_{i=1}^{n} A_i ) (\prod_{i=1}^{n} A_i) ^{p-1}
			\right]
			\right |
			\le \frac{1}{2|p|} \mathbb E \left[\prod_{i=1}^{n} A_i^p \mathds{1}_{\Omega_i} \right]
			\]
		\end{lemma}
		
		Now we are ready for the proof of Proposition \ref{propo: time delay SHF}.
		\begin{proof}[Proof of Proposition \ref{propo: time delay SHF}]
			We have 
			\begin{align}\label{eq: neg-1}
				\mathbb E [ M_n^p] \ge 
				\mathbb E [\mathds{1}_{\Omega} M_n^p]
				\ge  \mathbb E [\mathds{1}_{\Omega} (\prod_{i=1}^n A_i )^p] + 
				p \mathbb E \left[
				\mathds{1}_{\Omega} (M_n -\prod_{i=1}^n A_i )(\prod_{i=1}^n A_i)^{p-1}
				\right]
			\end{align}
			where we use $x^p\ge y^p +p(x-y)y^{p-1} $ for any $p<0$, $x\ge 0$ and $y>0$. 
			
			By applying Lemma \ref{lemma: EA_i^p 1_omega_i >...} and Lemma \ref{lemma: ratio} to \eqref{eq: neg-1}, we obtain
			\begin{align*}
				\bbE [M_n^p]
				\ge 
				\frac{1}{2}
				\bbE \big[ \ind_{\gO} (\prod_{i=1}^n A_i )^p \big]
				\ge C_\ell  N^{\frac{p(p-1)}{2}}
				\ge C_\ell  \big( \log \frac{1}{\eps} \big)^{\frac{p(p-1)}{2}}
			\end{align*}
			where the last inequality above follows from $N = \big\lfloor \frac{\log \frac{1}{\eps}}{\log \frac{1}{b}} \big\rfloor$. Finally, notice that 
			$(M_n)_n$ is a nonnegative martingale by Lemma \ref{lemma: MG of SHF}. Therefore, for any $p<0$, we have $\bbE[M_N^p]\ge \bbE [M_n^p]$, which finishes the proof.
		\end{proof}

		In the rest of this subsection, we prove Lemma \ref{lemma: EA_i^p 1_omega_i >...} and Lemma \ref{lemma: ratio}. To start with, we first present a series of moment estimates. The following lemma is from \cite{GuTsai_2026_loglogCLT}. It concerns moment estimates of $V_i$ and $B_I$. Estimate \eqref{eqL: E U_i^4 <} is the case of the fourth moment in \cite[(B.22)]{GuTsai_2026_loglogCLT}, and \eqref{eqL: E B_I^2<} is \cite[(3.40)]{GuTsai_2026_loglogCLT}. 
		\begin{lemma}\label{prop:E U_i^4 <}
			There are $C$, $b_0>0$ small enough, and $\ell_0\in \mathbb Z_+$ such that, for all $b\in (0,b_0)$ and $\ell\ge\ell_0$,
			\begin{equation}\label{eqL: E U_i^4 <}
				\E V_i^4\le \frac{C}{(N-i)^2}.
			\end{equation}
			Moreover, for every interval $I=\{m,\cdots,r\}\subset\{1,\cdots,n\}$ with $|I|\ge2$,
			\begin{equation}\label{eqL: E B_I^2<}
				\E B_I^2
				\le
				\frac{C^{|I|}}
				{(\log b^{-1} )^{|I|-1}(N-r)^2}.
			\end{equation}
		\end{lemma}
		
		Next, we give a second moment estimate that is used to get the correct exponent $p(p-1)/2$ in \eqref{eqL: E Z_epsilon^2 ,1 ^p >}.
		\begin{lemma}\label{lemma: Block variance}
			For every $c_0>0$, there are $C=C(c_0)>0$, $b_0\in(0,1/2)$ and $\ell_0\in \mathbb Z_+$ such that, for all $|\theta|\le c_0$, , $0<b\le b_0$, $\ell\ge\ell_0$, $1\le i\le N-\ell$, and sufficiently small $\epsilon$,
			\begin{equation}\label{eqL: E U_i^2 =...}
				\Big|\E V_i^2 - \frac{1}{N-i}\Big|
				\le \frac{C}{(N-i)^2 }.
			\end{equation}
		\end{lemma}

		\begin{proof}
			By the second-moment formula for SHF, we have 
			\[
			\begin{aligned}
				\mathbb E [V_i^2] = &
				4\pi \int _{0<s<s'<t_i-t_{i-1} } G_\theta(s'-s) \left[
				\int_{\mathbb R^2} g_{t_{i-1} + s} (y)^2 \td y
				\right] \td s \td s'\\
				=& \int _{0<s<s'<t_i-t_{i-1} } G_\theta (s'-s) \frac{1}{t_{i-1} + s} \td s \td s'\\
				=& \int _0^{t_i-t_i b^2} G_\theta (u) \log \left(\frac{t_i-u}{ b^2 t_i }\right) \td u
			\end{aligned}
			\]
			where we have changed variable $u:= s'-s$, complete the integral over $s$ and recall that $t_{i-1} = b^2 t_i$. 	
			By \eqref{eqSHF: G_theta (t small)}, we have $G_\theta(u)=\frac{1}{u(\log u^{-1})^2}(1+O_{c_0}(\frac{1}{\log u^{-1}})$, uniformly for $|\theta|\le c_0$. Thus
			\begin{align}
				\mathbb E [V_i^2] 
				& = 
				\int_0^{t_i(1-b^2)} \frac{1}{u |\log u|^2}
				\left(
				1+O_{c_0}\Big(\frac{1}{\log u^{-1}}\Big)
				\right) 
				\log \left( \frac{t_i-u}{b^2 t_i} \right) \td u \notag \\
				& =
				2\int_0^{t_i(1-b^2)} \frac{1}{u |\log u|^2}
				\left(
				1+O_{c_0}\Big(\frac{1}{\log u^{-1}}\Big) \right) 
				\log \frac{1}{b} \,\td u \label{int1} \\
				& \quad +
				\int_0^{t_i(1-b^2)} \frac{1}{u |\log u|^2}
				\left(
				1+O_{c_0 }\Big(\frac{1}{\log u^{-1}}\Big) \right) 
				\log \big( 1-\frac{u}{t_i} \big) \, \td u \label{int2} 
			\end{align}
			By taking $\ell_0$ sufficiently large, $u\le t_i=b^{2(N-i)}\le b^{2l}$ can be made arbitrarily small uniformly for $\ell >\ell_0$ and $0\le i\le N-\ell$. Hence the replacement using \eqref{eqSHF: G_theta (t small)} is justified. Due to the same reason, we can ensure  
			\begin{align}\label{est3}
				1+O_{c_0}\Big(\frac{1}{\log u^{-1}}\Big) \le 3,\quad \forall u\in [0,t_i(1-b^2) ].
			\end{align}
			We estimate \eqref{int1} as following: 
			\begin{equation}\label{eqL: mid I}
				\eqref{int1} = 2\log \frac{1}{b} \left[
				\frac{1}{|\log (t_i(1-b^2) )|} 
				+ O \Big( \frac{1}{ |\log (t_i(1-b^2) )|^{2}} \Big)
				\right].
			\end{equation}
			From $t_i= b^{2(N-i)}$,
			\begin{align*}
				\frac{1}{|\log (t_i(1-b^2) )|} 
				= \frac{1}{\log \frac{1}{t_i} + \log \frac{1}{1-b^2}}
				= \frac{1}{ 2(N-i) \log \frac{1}{b} + \log \frac{1}{1-b^2}}.
			\end{align*}
			Plugging it  into \eqref{eqL: mid I} we finish the estimate of \eqref{int1}. For \eqref{int2}, by \eqref{est3} we have 
			\[
			|\eqref{int2}|\le 3 \int_0^{t_i (1-b^2) } \frac{| \log (1-u/t_i) | }{u|\log u|^2} \td u\le \frac{3}{| \log (t_i)|^2} \int_0^{1-b^2 } \frac{|\log (1-v) |}{v} \td v\le \frac{C}{|N-i|^2}.
			\]
		\end{proof}
		
		\begin{remark}\label{rmk:early segment}
			We discard the segments after $n = N-l$ because Lemma \ref{prop:E U_i^4 <} and Lemma \ref{lemma: Block variance} are true only when $\ell_0$ is sufficiently large. For example, \eqref{eqSHF: G_theta (t small)} is only valid for sufficiently large $\ell$. 
		\end{remark}
		
		Now we close the subsection by the proof of Lemma \ref{lemma: EA_i^p 1_omega_i >...} and Lemma \ref{lemma: ratio}.
		
		\begin{proof}[Proof of Lemma \ref{lemma: EA_i^p 1_omega_i >...}]
			We will use the following basic inequality: fix any $p<0$, there exists $C_p>0$ such that
			\begin{equation}\label{eqL: (1+u)^p >...}
				(1+u)^p \ge 1+pu +\frac{p(p-1)}{2}u^2 -C_p |u^3|,\quad \forall |u|<\frac{1}{2}.
			\end{equation}
			By taking $u=V_i$ on the event $\Omega _i$ and then multiplying by $\mathds{1}_{\Omega _i }$, we have 
			\[
			\begin{aligned}
				\mathbb E [A_i^p \mathds{1}_{\Omega_i}] = &\mathbb E [(1+V_i)^p \mathds{1}_{\Omega_i}]\\
				\ge & \mathbb P(\Omega_i) + p\mathbb E [V_i \mathds{1}_{\Omega_i}] +\frac{p(p-1)}{2} \mathbb E [V_i^2  \mathds{1}_{\Omega_i}] -C_p \mathbb E [|  V_i  | ^3 \mathds{1}_{\Omega_i}].  
			\end{aligned}
			\]
			We estimate the r.h.s term by term. For the first term,
			\[
			\mathbb P(\Omega_i) =1-\mathbb P (|V_i| >1/2)\ge 1-16 \mathbb E [|V_i|^4] \ge 1-C (N-i)^{-2}
			\] 
			where the last inequality is due to \eqref{eqL: E U_i^4 <}. 
			
			For the second term, by $\mathbb E V_i =0$ and \eqref{eqL: E U_i^4 <}, we have 
			\begin{equation}\label{eqL: mid 1}
				| \mathbb E [V_i \mathds{1}_{\Omega_i} ] | = 
				| \mathbb E [V_i \mathds{1}_{\Omega _i ^c}] |\le 
				\mathbb E [|V_i|\mathds{1}_{|V_i| \ge 1/2}]\le 8 \mathbb E [V_i^4]\le \frac{C}{(N-i)^2 },
			\end{equation}
			For the third term, 
			\[
			\mathbb E [V_i^2 \mathds{1}_{\Omega_i}] = \mathbb E [V_i^2] - \mathbb E [V_i^2 \mathds{1}_{\Omega_i^c}]
			\]
			Similar to \eqref{eqL: mid 1}, we get $|\mathbb E [V_i^2 \mathds{1}_{\Omega_i^c}]|\le C/(N-i)^2$. Together with \eqref{eqL: E U_i^2 =...}, we have 
			\[
			\mathbb E [V_i^2 \mathds{1}_{\Omega_i}] \ge  \frac{1}{N-i} - \frac{C}{(N-i)^2}.
			\]
			For the fourth term, we apply H\"{o}lder's inequality and by \eqref{eqL: E U_i^4 <}, we obtain
			\[
			\begin{aligned}
				\mathbb E [|V_i| ^3 \mathds{1}_{\Omega _i}] 
				\le \big( \bbE V^4_i \big)^\frac{3}{4}
				\le \frac{C}{(N-i)^{\frac{3}{2}}}.
			\end{aligned}
			\]
			Summing it up gives \eqref{eqL: E A_i^p 1_onega_i >}. We also mention that, by the lower bound in \eqref{eqL: E A_i^p 1_onega_i >}, we can take sufficiently large $\ell $ so that $\mathbb E [A_i^p \mathds{1}_{\Omega_i}] \ge 1$. 
			
			Finally, we prove \eqref{eqL: E prod A_i Omega}. Take $\ell _0 $ large so that when $\ell \ge \ell_0$, we have  
			$$\frac{p(p-1)}{2(N-i)} -\frac{C}{(N-i) ^{3/2}} \in (0,\frac{1}{2})$$ 
			for all $1\le i \le N- \ell$. Then by \eqref{eqL: E A_i^p 1_onega_i >}, independence of $A_i$ and $\log (1+x)\ge x-x^2$ for all $x\in (0,\frac{1}{2})$, we have
			\[
			\begin{aligned}
				\log \mathbb E \left [\mathds{1}_{\Omega} (\prod_{i=1}^n A_i )^p \right ] = &\log \prod_{i=1}^n \mathbb E \left[  A_i^p\mathds{1}_{\Omega_i} \right]=
				\sum_{i=1}^n \log \mathbb E \left[ \mathds{1}_{\Omega_i} A_i^p \right]\\
				\ge & 
				\sum _{i=1}^{n} \frac{p(p-1)}{2 (N-i)} - \sum _{i=1}^{n} \frac{C}{(N-i)^{\frac{3}{2}}} \\ 
				\ge &
				\sum _{i=1}^{n} \frac{p(p-1)}{2 (N-i)} -C 
				\ge 
				\frac{p(p-1)}{2} \log \frac{N-1}{\ell} -C 
			\end{aligned}
			\]
			which proves \eqref{eqL: E prod A_i Omega}. 
			
		\end{proof}

		\begin{proof}[Proof of Lemma \ref{lemma: ratio}]
			From the decomposition \eqref{eqL: M_n expansion} of $M_n$ and by the fact that the SHF has independent increments, we get 
			\begin{equation}\label{eqL: negative moments, ratio by omega}
				\begin{aligned}
					\frac{\left | 
						\mathbb E \left[
						\mathds{1}_{\Omega} (M_n - \prod_{i=1}^n A_i ) (\prod_{i=1}^n A_i) ^{p-1}
						\right] 
						\right |}
					{\bbE [\ind_{\gO} \prod_{i=1}^n  A_i^p] }
					= & 
					\frac{\left | 
						\sum_{\substack{\mathcal I \in \mathfrak I_n  } } 
						\mathbb E [\prod_{I \in \mathcal I} B_I \prod_{i\in I } A_i^{p-1} \mathds{1}_{\Omega _i} ] \mathbb E [\prod_{i\notin U(\mathcal I)} A_i^p \mathds{1}_{\Omega _i} ]
						\right |}
					{\prod_{i=1}^n \bbE [\ind_{\gO_i}  A_i^p] }
					\\
					\le & 
					\sum_{\substack{\mathcal I \in \mathfrak I_n  } } \prod_{I \in \mathcal I}
					\frac{
						\Big| \bbE \Big[ B_I \prod_{i\in I} A_i^{p-1} \ind_{\gO_i} \Big]\Big|
					}{
						\prod_{i\in I} \bbE [\ind_{\gO_i}  A_i^p] 
					} \\ 
					\le &
					\sum_{\substack{\mathcal I \in \mathfrak I_n  } } \prod_{I \in \mathcal I}
					\left| \mathbb E \left[ B_I \prod_{i\in I} A_i^{p-1} \mathds{1}_{\Omega_i} \right] \right |
					=:
					\sum_{\substack{\mathcal I \in \mathfrak I_n  }} \prod_{I \in \mathcal I}
					w_I,
				\end{aligned}
			\end{equation}
			where in the last inequality we use $\bbE[A_i^p \ind_{\gO_i}]\ge 1$ from \eqref{eqL: E A_i^p 1_onega_i >}. 
			Recall we have the decomposition $A_i=V_i +1$. Set 
			$$h_i:= (1+V_i)^{p-1} \mathds{1}_{\Omega_i} -1,$$ 
			then $|h_i|\le C_p |V_i|$ for some $C_p>0$ due to our definition $\Omega_i := \{|V_i| <1/2\}$. Indeed, if $|V_i| \ge 1/2$, $|h_i| =1 \le 2|V_i|$. If $|V_i| \le 1/2$, by expanding $(1+V_i)^{p-1}$, we have
			$$|h_i| \le \max_{|x|\le 1/2} (1-p) (1+x)^{p-2} |V_i| = (1 - p ) 2^{2-p} |V_i|.$$ 
			
			For any interval $I =\{m,m+1,\cdots,r\}\subset \{1,\cdots,n\}$ with $|I| \ge 2$, we have 
			\[
			\begin{aligned}
				w_I= \left 	|\mathbb E \left[ B_I \prod_{i\in I} A_i^{p-1} \mathds{1}_{\Omega_i} \right] \right |
				= \left |
				\mathbb E \left[
				B_I \prod_{i\in I} (h_i +1)
				\right] \right | 
				= 
				\left | 
				\sum_{ \substack{J \subset I \\ J\neq\emptyset }} \mathbb E \left[ B_I 
				\prod_{i\in J} h_i
				\right]
				\right |
			\end{aligned}
			\]
			where we use the polynomial expansion and $\bbE[B_I] = 0$ in the last inequality. By Cauchy-Schwartz and independence, we further have 
			\[
			\begin{aligned}
				w_I \le & \sqrt{\mathbb E [|B_I|^2 ]} \sum_{ \substack{J \subset I \\ J\neq\emptyset }}  \prod_{i\in J} \sqrt {
					\mathbb E \left[
					|h_i|^2 
					\right]
				}
				=   \sqrt{\mathbb E [|B_I|^2 ]}  
				\left[
				\prod_{i\in I} (\sqrt {
					\mathbb E \left[
					|h_i|^2 
					\right]} +1)-1
				\right]
			\end{aligned}
			\] 
			By \eqref{eqL: E B_I^2<} and our setting $I= \{m,m+1,\cdots,r\}$, we have 
			\[
			\mathbb E [|B_I|^2 ] \le \frac{C^{|I|} }{(\log b^{-1} )^{|I|-1} (N-r)^2}.
			\] 
			Also, we have discussed earlier that $|h_i|\le C_p|V_i|$; hence $\mathbb E \left[
			|h_i|^2 
			\right]  \le C_h^2 \mathbb E [V_i^2]\le \frac{C}{N-i}$ by \eqref{eqL: E U_i^2 =...}, and  
			\[
			\begin{aligned}
				\prod_{i\in I} (\sqrt {
					\mathbb E \left[
					|h_i|^2 
					\right]} +1)-1 \le &\prod_{i\in I} \left(\frac{C}{\sqrt{N-i} } +1 \right) -1 
				\le \left(\frac{C}{\sqrt{N-r}} +1 \right)^{|I|} -1 \\
				& \le  \frac{1}{\sqrt{N-r}} \sum_{i=1}^{|I|} \binom{|I|}{i} C^{i}
				\le \frac{(1+C)^{|I|}}{\sqrt{N-r}},
			\end{aligned}
			\]
			Thus we have, for an interval $I = \{m,\cdots, r\}$,
			\[
			w_I \le \frac{(1+C) ^{|I| }}{(\log b^{-1} )^{(|I|-1) /2 } (N-r)^{3/2}}.
			\]
			Let us go back to \eqref{eqL: negative moments, ratio by omega}: Since
			\begin{align}\label{eq: neg bdd last}
				\sum_{\substack{\mathcal I \in \mathfrak I_n \\ \cI\neq \emptyset}} 
				\prod_{I \in \mathcal I}
				w_I
				\le
				\prod_{\substack{I \subset \{1,\cdots,n\} \\ I\text{ interval, } |I|\ge 2 }}
				(1+ w_I)-1
				\le \exp\Big\{ \sum_{\substack{I \subset \{1,\cdots,n\} \\ I\text{ interval, } |I|\ge 2 }} w_I  \Big\}-1,
			\end{align}
			by summing over the right end points and length of these intervals, we have
			\begin{align}\label{S bdd}
				\sum_{
					I = \{m,..,r\}\in \{ 1,...,n=N-\ell \}\text{ with }|I| =l\ge 2
				} w_I 
				\le 
				\sum^{n }_{r=2} \frac{1}{(N-r)^{\frac{3}{2}}} 
				\sum^\infty_{l=2} \frac{C^l}{\big(\log \frac{1}{b}\big)^{\frac{l-1}{2}}}
				\le
				\frac{C}{\sqrt{\ell  \, \log \frac{1}{b}}}
			\end{align}
			where the last inequality is achieved by taking a fixed $b$ small enough. 
			Finally, by choosing $\ell_0$ and thus $\ell \ge \ell_0$ sufficiently large in \eqref{S bdd}, \eqref{eq: neg bdd last} can be bounded above by $\frac{1}{2|p|}$, which finishes the proof. 
		\end{proof}
		
		\section{Applications to GMC and mSHE} \label{section: GMC and mSHE}
		In this section, we discuss the applications of the moment comparison tools in Section \ref{section: moment comparison} for subcritical Gaussian multiplicative chaos and mollified stochastic heat equations. 
		
		\subsection{Gaussian multiplicative chaos}\label{subsection: GMC}
		In this subsection, we prove the finiteness of negative moments of subcritical Gaussian multiplicative chaos (GMC) which was proven, for example, in \cite{RobertVargas2010_GMCrevisited} and \cite{DS2010_liouville_and_KPZ}. To the best of our knowledge, the existing proofs rely on the 'branching structure' of GMC and a resulted bootstrap inequality of the left tail. Here we reprove it as a direct application of the Ehrhard-type moment comparison inequality of \cite{Chen2026_Ehrhard_convex1}. As a benefit, we do not only recover the finiteness of negative moments directly, but also obtain a tractable upper bound.

		We start with the basics of GMC. Fix the dimension $d$ and a function 
		\[
		\rho\in C_c^\infty (\mathbb R^d),\quad  \|\rho\|_{L^2(\mathbb R^d)} =1,\quad \rho \ge 0.
		\]
		Let $W(\td y,\td r)$ be a $d+1$ White noise. Consider the Gaussian field 
		\[
		X_\epsilon(x) := \int_1^{\epsilon^{-1}} \int_{\mathbb R^d} s^{\frac{d-1}{2} } \rho( s(x-y) ) W(\td y, \td s).
		\]
		Or equivalently, $X_\epsilon(x) = W(h_{\epsilon ,x})$ with 
		\[
		h_{\epsilon,x } (s,y): = \mathds{1}_{r\in [1,\epsilon ^{-1} ] } s^{\frac{d-1}{2} } \rho(s(x-y) ). 
		\]
		For any $\gamma \in (0,\sqrt{2d})$, set the log-normal measure 
		\[
		M_{\epsilon}^\gamma (\td x) := 
		e^{\gamma X_\epsilon (x) - \frac{\gamma ^2 }{2} \mathbb E [X_\epsilon (x) ^2 ] } \dd x.
		\]
		
		\begin{theorem}[refer to \cite{RV_2014_GMCreview}]
			For any $\gamma \in (0,\sqrt{2d})$, $M_{\epsilon}^\gamma (\td x)$ converges a.s. and in $L^1$ to some random measure $M^\gamma$ under the vague topology. Especially, for any non-negative $\phi \in C_c(\mathbb R^d)$ with $\| \phi \| _{L^1 (\mathbb R^d) } >0$, $M^\gamma (\phi) >0$ almost surely.
		\end{theorem}
		
		The limiting measure $M^\gamma$ is called a subcritical GMC with exponent $\gamma$. From the positivity of $M^\gamma (f)$, we also get the a.s. convergence of $\log M^\gamma_\eps(f)$ to $\log M^\gamma (f)$. From Theorem 2.11 of \cite{RV_2014_GMCreview}, for any $0<p<\frac{2d}{\gamma^2}$, we have $\mathbb E [M^\gamma (f) ^p] <\infty$. 
		
		To employ the Ehrhard-type moment comparison result in Section \ref{section: moment comparison}, we prove in the following lemma that $\log M^\gamma(\phi)$ is Ehrhard-type. 
		
		\begin{lemma}\label{lemma: subGMC is Ehrhard-type}
			Fix $d\in \mathbb Z_+$ and $\gamma \in (0,\sqrt{2d})$. For any non-negative $\phi \in C_c(\mathbb R^d)$ with $\| \phi \| _{L^1 (\mathbb R^d) } >0$, $\log M^\gamma (\phi )$ is an Ehrhard-type random variable (see Definition \ref{def: Ehrhard-type random variables}). 
		\end{lemma}
		
		\begin{proof}
			According to Lemma \ref{lemma: Ehrhard type is close}, it suffices to prove that $\log M^\gamma _\epsilon (\phi)$ is Ehrhard-type. Take an orthonormal  basis $\{ e_k: k=1,...\}$ of $H:= L^2(\mathbb R^d \times [1,\infty ))$.
			Set 
			\[
			M^\gamma _{\epsilon ,N}(\phi ):= 
			\int_{\mathbb R^d} \phi(x) 
			e^{\gamma \sum_{k=1} ^N W (e_k) \left\langle h_{\epsilon,x} , e_k\right\rangle_H -\frac{\gamma ^2 }{2} \sum_{k=1}^N \left\langle h_{\epsilon,x} , e_k\right\rangle_H ^2  } \td x
			\]
			which converges a.s. and in $L^1$ to $M^\gamma _\epsilon (\phi)$. Indeed, since $\{W(e_k):k=1,...\}$ is a family of independent standard Gaussian random variables and $W(h_{\epsilon,x}) = W (P_Nh_{\epsilon,x} ) + W ( (I-P_N)h_{\epsilon,x} )$ with the projection $P_Nh_{\epsilon,x} (y) = \sum_{k=1}^N \left\langle h_{\epsilon,x} , e_k\right\rangle_H e_k$, the claimed convergence follows from 
			$
			M^\gamma _{\epsilon ,N}(\phi ) = \mathbb E \left[
			M^\gamma _{\epsilon }(\phi ) | \mathcal F_N
			\right]
			$  
			where $\mathcal F_N$ is the filter generated by $\{ W(e_1),...W(e_N)\}$. 
			
			Note that $\log M^\gamma _{\epsilon ,N}(\phi )= F_N( W(e_1), ...,W(e_N) ) $ with 
			\[
			F_N(\mathbf z)
			=
			F_N(z_1,...,z_N):= 
			\log \int_{\mathbb R^d} \phi(x) 
			e^{\gamma \sum_{k=1} ^N z_k \left\langle h_{\epsilon,x} , e_k\right\rangle_H -\frac{\gamma ^2 }{2} \sum_{k=1}^N \left\langle h_{\epsilon,x} , e_k\right\rangle_H ^2  } \td x
			\]
			By H\"{o}lder's inequality, $F_N(\gl \mathbf z+ (1-\gl) \hat{ \mathbf z}) \le \lambda F_N(\mathbf z) + (1-\lambda ) F_N(\hat{\mathbf z})$ for any $\lambda \in (0,1)$. As a result, $\log M^\gamma _{\epsilon ,N}(\phi )$ is a convex functional of the Gaussian vector $( W(e_1),...,W(e_N) )$ and thus Ehrhard-type by Lemma \ref{lemma: Ehrhard and convex Gaussian}. Therefore, the limit $\log M^\gamma _\epsilon (\phi)$ is also Ehrhard-type and the proof is finished.   
		\end{proof}
		
		As a direct application of Theorem \ref{thm: Ehrhard moment comparison ineq} and Lemma \ref{lemma: subGMC is Ehrhard-type}, we have the following bound for the negative moments of GMC:
		\begin{theorem}\label{thm: GMC, negative moments}
			Consider the same setting and notations as Lemma \ref{lemma: subGMC is Ehrhard-type}. Assume $\|\phi \|_{L^1(\mathbb R^d )}=1$. Then, for any fixed $p\in (1, \frac{2d}{\gamma^2})$, we have 
			\begin{equation}\label{eqGMC: negative moments <}
				\mathbb E \left[ ( M^\gamma (\phi) )^q\right] \le 
				\left(
				\mathbb E \left[ ( M^\gamma (\phi ) )^p\right]
				\right)^{\frac{q(q-1) }{p(p-1)}} <\infty ,\quad \forall q<0. 
			\end{equation}
		\end{theorem}
		
		\begin{remark}
			The assumption $\|\phi \|_{L^1(\mathbb R^d)}=1$ in Theorem \ref{thm: GMC, negative moments} is only for convenience. For general $\phi$, we can replace $M^\gamma (\phi)$ by $\frac{M^\gamma (\phi )}{ \|\phi \|_{L^1 (\bbR^d)}}$ and then proceed with the same argument. 
		\end{remark}
		
		\subsection{Mollified 2d critical stochastic heat equation}\label{subsection: mSHE}
		Parallel to the DPRE framework introduced in Section \ref{section: preliminary, SHF}, the $2d$ critical SHF can also be obtained as the weak limit of the solution of the mollified critical stochastic heat equation \eqref{eqSHE: mSHE} under a suitable choice of $\gb_\eps$; see \cite{Tsai2024_SHF_by_moments}.
		
		Assume $\xi$ is a $2+1$ time-space white noise. Take a smooth mollifier 
		$\phi_\epsilon (x):= \epsilon ^{-2}\phi(\epsilon ^{-1}x)$ with $\epsilon \in (0,\frac{1}{2})$, non-negative $\phi\in C_c^\infty (\mathbb R^2)$ with $\| \phi \| _{L^1 (\mathbb R^2)} =1$ and write the spatially mollified white noise $\xi_\epsilon (\td t ,x):= \int_{\mathbb R^2} \phi_\epsilon (x-y) \xi(\td t, \td y )$. Especially, we have the covariance 
		\[
		\mathbb E \left[
		\xi_\epsilon (\td t, x_1) \xi_\epsilon (\td s, x_2) 
		\right] 
		=  \delta _{t-s} \dd s \dd t \int_{\mathbb R^2} \phi_\epsilon (x_1-x_2+y) \phi_\epsilon (y) \td y 
		=: \delta _{t-s}  \dd s \dd t R_\epsilon (x_1-x_2 ) .
		\]
		Consider the mollified stochastic heat equation (mSHE)
		\begin{equation}\label{eqSHE: mSHE}
			\td u_\epsilon (t,x) = \frac{1}{2}\Delta u_\epsilon (t,x) \td t+ \beta_\epsilon u_\epsilon \xi_\epsilon  (\td t, x) 
		\end{equation}
		with $\beta_\epsilon  \in \mathbb R$ for each $\epsilon \in (0,\frac{1}{2})$. 
		
		For any non-negative and non-zero $f$, $g$, write
		$u_\epsilon (t;f,\cdot)$ as the solution for the mSHE above with initial condition $u_\epsilon (0,x) = f(x)$ and 
		$u_\epsilon (t; f,g):= \int _{\mathbb R^2} u_\epsilon(t;f,x) g(x) \td x$. Then, by Feynman-Kac formula (see \cite{BC1998_secondmomentSHE}), 
		\begin{equation}\label{eqSHE: F-K}
			u_\epsilon (t;f,g)= \int_{\mathbb R^2} g(x)
			E_x\left[
			e^{\beta_\epsilon \int_0^t  \xi_\eps(\td r, B_{t-r} ) - \frac{\beta_\epsilon ^2}{2} R_\epsilon(0) t} f(B_t)
			\right] \td x.
		\end{equation}
		Especially, we get the a.s. positivity of the solution when $\| f\|_{L^1(\mathbb R^2)} >0$ and $\| g\|_{L^1(\mathbb R^2)} >0$. 
		For convenience of discussion, we give a GMC-like representation for \eqref{eqSHE: F-K}: Define $\Omega := C( [0,t]; \mathbb R^2)$ and 
		$H:= L^2 ([0,t]\times \mathbb R^2)$ and write $h_\eps(\cdot) : \Omega \rightarrow H$ with $h_\eps(\gamma) (s,x) := \phi_\epsilon (\gamma_{t-s} -x )$ for $\gamma\in \Omega$, then
		\begin{equation}\label{eqSHE: F-K 2}
			u_\epsilon(t;f,g) = \int_\gO \mu_{f,g}(t,\td \gamma) e^{\beta_\epsilon \xi (h_\eps(\gamma)) - \frac{\beta_\epsilon^2 }{2} \mathbb E [\xi (h_\eps(\gamma))^2  ]}
		\end{equation}
		where $\xi(h_\eps(\gamma)) = \int_{\bbR_+\times \bbR^2} h_\eps(\gamma)(t,x) \xi(\dd t,\dd x)$, and $\mu_{f,g}(t,\cdot) $ is a positive finite measure on $\Omega $ defined by
		\begin{align*}
			\mu_{f,g}(t,\dd \gamma)
			:=
			\int_{\bbR^2}g(x) f(\gamma_t) 
			P_x( \dd \gamma)\, \dd x,
		\end{align*} 
		where $P_x$ is the Wiener measure started at $x$.
		
		\begin{lemma}\label{lemma: mSHE solution is Ehrhard-type}
			For any $\epsilon \in (0,\frac{1}{2})$, $f\equiv 1$, $g\in L^1(\mathbb R^2)$, $\log u_\epsilon(t;f,g)$ with $u_\epsilon$ defined in \eqref{eqSHE: F-K 2} is an Ehrhard-type random variable defined in Definition \ref{def: Ehrhard-type random variables}.
		\end{lemma} 
		Since the proof follows the same argument as that of Lemma \ref{lemma: subGMC is Ehrhard-type}, we omit the details.

		\appendix
		\section*{Acknowledgment and AI tool disclosure}
		Z.X. is supported by the China Scholarship Council–University of Warwick Award. We thank Nikos Zygouras for his valuable and useful comments and suggestsions. In particular, Z.X. thanks him for discussions on the blowing-up phenomenon of Lipschitz constants of the free energy in the critical two-dimensional DPRE, which provided the initial motivation for this project. We also thank Francesco Caravenna and Rongfeng Sun for their helpful comments. 
		
		When studying the left tail of the critical two-dimensional DPRE, and in particular seeking a potential substitution for the Gaussian concentration inequality used heavily in earlier work on DPRE, the authors used GPT-5.6 in Thinking mode as an exploratory tool to find the results of Ehrhard inequality and Gaussian convexity property. The proof idea of Subsection \ref{subsection: Local negative moments lower bound} and especially the sharp second moment estimate in Lemma \ref{lemma: Block variance} were suggested by GPT-5.6 pro. However, we emphasize that the framework there based on a geometric decomposition of locally averaged SHF has appeared in \cite{GuTsai_2026_loglogCLT} and related earlier work before our use of AI. 
		
		\bibliographystyle{alpha}
		\bibliography{Ehrhard_SHF}
	\end{document}